\pdfoutput=1
\RequirePackage{ifpdf}
\ifpdf 
\documentclass[pdftex]{sigma}
\else
\documentclass{sigma}
\fi

\numberwithin{equation}{section}

\newtheorem{Theorem}{Theorem}[section]
\newtheorem{Corollary}[Theorem]{Corollary}
\newtheorem{Lemma}[Theorem]{Lemma}
\newtheorem{Proposition}[Theorem]{Proposition}
 { \theoremstyle{definition}
\newtheorem{Definition}[Theorem]{Definition}
\newtheorem{Example}[Theorem]{Example} }

\begin{document}

\allowdisplaybreaks

\renewcommand{\thefootnote}{$\star$}

\newcommand{\arXivNumber}{1511.06721}

\renewcommand{\PaperNumber}{033}

\FirstPageHeading

\ShortArticleName{Orthogonality Measure on the Torus for Vector-Valued Jack Polynomials}

\ArticleName{Orthogonality Measure on the Torus\\ for Vector-Valued Jack Polynomials\footnote{This paper is a~contribution to the Special Issue
on Orthogonal Polynomials, Special Functions and Applications.
The full collection is available at \href{http://www.emis.de/journals/SIGMA/OPSFA2015.html}{http://www.emis.de/journals/SIGMA/OPSFA2015.html}}}

\Author{Charles F.~DUNKL}

\AuthorNameForHeading{C.F.~Dunkl}

\Address{Department of Mathematics, University of Virginia,\\ PO Box 400137, Charlottesville VA 22904-4137, USA}
\Email{\href{mailto:cfd5z@virginia.edu}{cfd5z@virginia.edu}}
\URLaddress{\url{http://people.virginia.edu/~cfd5z/}}

\ArticleDates{Received November 26, 2015, in f\/inal form March 23, 2016; Published online March 27, 2016}	

\Abstract{For each irreducible module of the symmetric group on~$N$ objects there is a~set
of parametrized nonsymmetric Jack polynomials in~$N$ variables taking values in
the module. These polynomials are simultaneous eigenfunctions of a commutative
set of operators, self-adjoint with respect to certain Hermitian forms. These
polynomials were studied by the author and J.-G.~Luque using a Yang--Baxter
graph technique. This paper constructs a~matrix-valued measure on the $N$-torus
for which the polynomials are mutually orthogonal. The construction uses
Fourier analysis techniques. Recursion relations for the Fourier--Stieltjes
coef\/f\/icients of the measure are established, and used to identify parameter
values for which the construction fails. It is shown that the absolutely
continuous part of the measure satisf\/ies a f\/irst-order system of dif\/ferential equations.}

\Keywords{nonsymmetric Jack polynomials; Fourier--Stieltjes coef\/f\/icients; matrix-valued measure; symmetric group modules}

\Classification{33C52; 42B10; 20C30; 46G10; 35F35}

\renewcommand{\thefootnote}{\arabic{footnote}}
\setcounter{footnote}{0}

\section{Introduction}

The Jack polynomials form a parametrized basis of symmetric polynomials. A~special case of these consists of the Schur polynomials, important in the
character theory of the symmetric groups. By means of a commutative algebra of
dif\/ferential-dif\/ference operators the theory was extended to nonsymmetric Jack
polynomials, again a parametrized basis but now for all polynomials in~$N$
variables. These polynomials are orthogonal for several dif\/ferent inner
products, and in each case they are simultaneous eigenfunctions of a~commutative set of self-adjoint ope\-ra\-tors. These inner products are invariant
under permutations of the coordinates, that is, the symmetric group. One of
these inner products is that of $L^{2}\big( \mathbb{T}^{N},K_{\kappa}(x) \mathrm{d}m(x) \big) $, where
\begin{gather*}
\mathbb{T}^{N} :=\big\{ x\in\mathbb{C}^{N}\colon \vert x_{j} \vert
=1,\, 1\leq j\leq N\big\} ,\\
\mathrm{d}m(x) =(2\pi) ^{-N}\mathrm{d}
\theta_{1}\cdots\mathrm{d}\theta_{N},\qquad x_{j}=\exp ( \mathrm{i}\theta
_{j} ) ,\qquad -\pi<\theta_{j}\leq\pi,1\leq j\leq N,\\
K_{\kappa}(x) =\prod_{1\leq i<j\leq N} \vert x_{i}
-x_{j} \vert ^{2\kappa},\qquad \kappa>-\frac{1}{N};
\end{gather*}
def\/ining the $N$-torus, the Haar measure on the torus, and the weight function
respectively. Beerends and Opdam~\cite{Beerends/Opdam1993} discovered this
orthogonality property of symmetric Jack polyno\-mials. Opdam~\cite{Opdam1995}
established orthogonality structures on the torus for trigonometric
polynomials associated with Weyl groups; the nonsymmetric Jack polynomials
form a special case. Details on the derivation of the norm formulae can be
found in the treatise by Xu and the author~\cite[Section~10.6.3]{Dunkl/Xu2014}. The
weight function~$K_{\kappa}$ turned out to be the square of the base state for
the Calogero--Sutherland quantum mechanical model of~$N$ identical particles
located at $x_{1},x_{2},\ldots,x_{N}$ on the circle with a $1/r^{2}$
potential. This means that the particles repel each other with a force
corresponding to a potential $C \vert x_{i}-x_{j} \vert ^{-2}$. See
Lapointe and Vinet~\cite{Lapointe/Vinet1996} for the construction of wavefunctions
in terms of Jack polynomials for this model. More recently Grif\/feth~\cite{Griffeth2010} constructed vector-valued Jack polynomials for the family
$G\left( n,p,N\right) $ of complex ref\/lection groups. These are the groups
of permutation matrices (exactly one nonzero entry in each row and each
column) whose nonzero entries are $n^{th}$ roots of unity and the product of
these entries is a $ ( n/p)^{\rm th}$ root of unity. The symmetric
groups and the hyperoctahedral groups are the special cases $G(1,1,N) $ and $G( 2,1,N) $ respectively. The term
``vector-valued'' means that the polynomials
take values in irreducible modules of the underlying group, and the action of
the group is on the range as well as the domain of the polynomials. The author~\cite{Dunkl2010} together with Luque~\cite{Dunkl/Luque2011} investigated the
symmetric group case more intensively. The results from these two papers are
the foundation for the present work.

Since the torus structure is such an important aspect of the theory of Jack
polynomials it seemed like an obvious research topic to f\/ind the role of the
torus in the vector-valued Jack case. Is there a matrix-valued weight function
on the torus for which the vector-valued Jack polynomials are mutually
orthogonal? Some explorations in the $N=3$ and $N=4$ situation showed that the
theory is much more complicated than the ordinary (scalar) case. For
two-dimensional representations the weight function has hypergeometric
function entries (see \cite{Dunkl2014}); this is quite dif\/ferent from the
rather natural product $\prod\limits_{1\leq i<j\leq N} \vert x_{i}-x_{j}\vert ^{2\kappa}$, a power of the discriminant.

In this paper we will produce a matrix-valued measure on the torus for which
the vector-valued nonsymmetric polynomials are mutually orthogonal. The result
applies to arbitrary irreducible representations of the symmetric groups. In
each case there is a permitted range of the parameter. We start with a concise
outline of the def\/initions and construction of the polynomials using the
Yang--Baxter graph technique in Section~\ref{section2}, based on~\cite{Dunkl2010} and~\cite{Dunkl/Luque2011}. Section~\ref{section3} contains the construction of the abstract
Hermitian form which is designed to act like an integral over the torus; that
is, multiplication by a coordinate function $x_{j}$ is an isometry. The method
is algebraic and based on the Yang--Baxter graph. In Section~\ref{section4} we use
techniques from Fourier analysis to produce the desired measure. The Section
begins by using the formulae from the previous sections to def\/ine the
hypothetical Fourier--Stieltjes coef\/f\/icients, def\/ined on~$\mathbb{Z}^{N}$ which
is the dual group of the torus, a multiplicative group, and then applying a
matrix version of a~theorem of Bochner about positive-def\/inite functions to
get the measure. There is an application of approximate-identity theory using
a Ces\`{a}ro kernel to construct a sequence of positive matrix-valued Laurent
polynomials which converges to the orthogonality measure.

Section~\ref{section5} develops a recurrence relation satisf\/ied by the Fourier--Stieltjes
coef\/f\/icients of the orthogonality measure. The relation allows an inductive
calculation for the coef\/f\/icients (but actual work, even with symbolic
computation software, may not be feasible unless the dimensions are reasonably
small), and it describes the list of parameter values (certain rational
numbers) for which the construction fails.

The scalar weight function on the torus
\begin{gather*}
K_{\kappa}(x) :=\prod\limits_{1\leq i<j\leq N}\big\{ (
x_{i}-x_{j} ) \big( x_{i}^{-1}-x_{j}^{-1}\big) \big\} ^{\kappa}
\end{gather*}
satisf\/ies a f\/irst-order dif\/ferential system,
\begin{gather*}
x_{i}\frac{\partial}{\partial x_{i}}K_{\kappa}(x) =\kappa
K(x) \sum_{j\neq i}\frac{x_{i}+x_{j}}{x_{i}-x_{j}}, \qquad 1\leq i\leq
N.
\end{gather*}
In Section~\ref{section6} we show that there is an analogous matrix dif\/ferential system
which is solved in a distribution sense by the orthogonality measure. We
outline a result asserting that the orthogonality measure restricted to the
complement of $\bigcup\limits_{1\leq i<j\leq N} \{ x\colon x_{i}=x_{j} \}
$ is equal to an analytic solution of the dif\/ferential system times the Haar
measure~$\mathrm{d}m$. Finally there is Appendix~\ref{appendixA} containing some technical
background results.

\section{Vector-valued Jack polynomials and the Yang--Baxter graph}\label{section2}

This is a summary of the def\/initions and results from \cite{Dunkl2010} and
\cite{Dunkl/Luque2011}. For $x= ( x_{1},\ldots,x_{N} ) \in\mathbb{C}^{N}$
the \textit{monomial} $x^{\alpha}:=\prod\limits_{i=1}^{N}x_{i}^{\alpha_{i}}$, $\alpha= ( \alpha_{1},\ldots,\alpha_{N} ) \in\mathbb{N}_{0}^{N}$,
$ \vert \alpha \vert :=\sum\limits_{i=1}^{N}\alpha_{i}$, $\mathbb{N}_{0}:= \{ 0,1,2,3,\ldots \} $ and~$\alpha$ is called a multi-index
or a composition of $ \vert \alpha \vert $. We denote two
distinguished elements by $\mathbf{0}= ( 0,0,\ldots,0 ) $, and
$\mathbf{1}= ( 1,1,\ldots,1 ) $. The \textit{degree} of $x^{\alpha
}$ is $ \vert \alpha \vert $, and a polynomial is a~f\/inite linear
combination of monomials. The linear space of all polynomials is denoted by~$\mathcal{P}$, and $\mathcal{P}_{n}=\operatorname{span} \{ x^{\alpha
}\colon \vert \alpha \vert =n \} $ is the subspace of polynomials
homogeneous of degree $n$. The specif\/ic polynomials considered here have
coef\/f\/icients in $\mathbb{Q} ( \kappa ) $ where $\kappa$ is
transcendental (indeterminate) but which will also take on certain real
values. The multi-indices $\alpha$ have an important partial order: let
$\alpha^{+}$ denote the nonincreasing rearrangement of~$\alpha$, for example
if $\alpha= ( 1,2,1,4 ) $ then $\alpha^{+}= ( 4,2,1,1 )
$. Let $\mathbb{N}_{0}^{N,+}$ denote the set of partition multi-indices, that
is, $\big\{ \lambda\in\mathbb{N}_{0}^{N}\colon \lambda_{1}\geq\lambda_{2}%
\geq\cdots\geq\lambda_{N}\big\}$.

\begin{Definition}
\begin{gather*}
\alpha \prec\beta~\Longleftrightarrow\sum_{j=1}^{i}\alpha_{j}\leq\sum
_{j=1}^{i}\beta_{j},\qquad 1\leq i\leq N,\qquad \alpha\neq\beta,\\
\alpha\vartriangleleft\beta \Longleftrightarrow ( \vert
\alpha \vert = \vert \beta \vert ) \wedge\big[ (
\alpha^{+}\prec\beta^{+} ) \vee ( \alpha^{+}=\beta^{+}\wedge
\alpha\prec\beta ) \big].
\end{gather*}
\end{Definition}

For example $(3,2,1) \vartriangleleft ( 0,2,4 )
\vartriangleleft ( 4,0,2 ) $, while $ ( 4,1,1 )$, $( 3,3,0) $ are not $\vartriangleleft$-comparable. The
\textit{symmetric group} $\mathcal{S}_{N}$, the set of permutations of
$\{ 1,2,\ldots,N\} $, acts on $\mathbb{C}^{N}$ by permutation of
coordinates. The action is extended to polynomials by $wp(x)
=p ( xw ) $ where $(xw) _{i}=x_{w(i)}$
(consider $x$ as a row vector and $w$ as a permutation matrix, $[w] _{ij}=\delta_{i,w(j) }$, then $xw=x[w]
$). This is a representation of $\mathcal{S}_{N}$, that is, $w_{1}(w_{2}p) (x) =( w_{2}p) (
xw_{1}) =p( xw_{1}w_{2}) =( w_{1}w_{2})
p(x) $ for all $w_{1},w_{2}\in\mathcal{S}_{N}$.

Furthermore $\mathcal{S}_{N}$ is generated by ref\/lections in the mirrors
$\{ x\colon x_{i}-x_{j}=0\} $ for $1\leq i<j\leq N$. These are
\textit{transpositions, }denoted by $(i,j) $, interchanging
$x_{i}$ and $x_{j}$. Def\/ine the $\mathcal{S}_{N}$-action on $\alpha
\in\mathbb{N}_{0}^{N}$ so that $(xw) ^{\alpha}=x^{w\alpha}$%
\begin{gather*}
(xw) ^{\alpha}=\prod_{i=1}^{N}x_{w(i) }
^{\alpha_{i}}=\prod_{j=1}^{N}x_{j}^{\alpha_{w^{-1}(j) }},
\end{gather*}
that is $(w\alpha) _{i}=\alpha_{w^{-1}(i) }$ (take
$\alpha$ as column vector, then $w\alpha=[w] \alpha$).

The \textit{simple reflections} $s_{i}:=(i,i+1) $, $1\leq i\leq
N-1$, generate $\mathcal{S}_{N}$. They are the key devices for applying
inductive methods, and satisfy the \textit{braid} relations:
\begin{gather*}
s_{i}s_{j} =s_{j}s_{i}, \qquad \vert i-j \vert \geq2;\\
s_{i}s_{i+1}s_{i} =s_{i+1}s_{i}s_{i+1}.
\end{gather*}

\looseness=-1
We consider the situation where the group $\mathcal{S}_{N}$ acts on the range
as well as on the domain of the polynomials. We use vector spaces (called
$\mathcal{S}_{N}$-modules) on which $\mathcal{S}_{N}$ has an irreducible
unitary (orthogonal) representation: $\tau\colon \mathcal{S}_{N}\rightarrow
O_{m} ( \mathbb{R} ) $ ($\tau(w) ^{-1}=\tau(w^{-1}) =\tau(w) ^{T}$). See James and Kerber~\cite{James/Kerber1981} for representation theory, including a modern discussion of
Young's methods. We will specify an orthogonal basis and the images
$\tau( s_{i}) $ for each $i$, which suf\/f\/ices for our purposes.
Identify $\tau$ with a partition of $N$: $( \tau_{1},\tau_{2},\ldots) \in\mathbb{N}_{0}^{N,+}$ such that $\vert \tau\vert
=N$. The length of $\tau$ is $\ell(\tau) =\max \{
i\colon \tau_{i}>0 \} $. There is a Ferrers diagram of shape $\tau$ (this
diagram is given the same name), with boxes at points $(i,j) $
with $1\leq i\leq\ell(\tau) $ and $1\leq j\leq\tau_{i}$. A~\textit{tableau} of shape $\tau$ is a~f\/illing of the boxes with numbers, and a
\textit{reverse standard Young tableau} (RSYT) is a f\/illing with the numbers
$ \{ 1,2,\ldots,N \} $ so that the entries decrease in each row and
each column. We exclude the one-dimensional representations corresponding to
one-row $(N) $ or one-column $( 1,1,\ldots,1) $
partitions, that is, we require $\dim V_{\tau}\geq2$. The \textit{hook-length}
of the node $(i,j) \in\tau$ is def\/ined to be
\begin{gather*}
\operatorname{hook} ( \tau;i,j ) :=\tau_{i}-j+\#\big\{ k\colon i<k\leq
\ell(\tau) \wedge j\leq\tau_{k}\big\} +1.
\end{gather*}
We will need the key quantity $h_{\tau}:=\operatorname{hook} ( \tau;1,1 )
=\tau_{1}+\ell(\tau) -1$, the maximum hook-length of the diagram.

\begin{Example}
Here are the Ferrers diagram, a (column-strict) tableau, and an RSYT, all of
shape $(5,3,2)$
\begin{gather*}
\begin{matrix}
\square & \square & \square & \square & \square\\
\square & \square & \square & & \\
\square & \square & & &
\end{matrix}
, \qquad
\begin{bmatrix}
0 & 0 & 1 & 2 & 3\\
1 & 2 & 2 & & \\
2 & 4 & & &
\end{bmatrix}
,\qquad
\begin{bmatrix}
10 & 7 & 4 & 2 & 1\\
9 & 6 & 3 & & \\
8 & 5 & & &
\end{bmatrix}.
\end{gather*}
\end{Example}

Denote the set of RSYT's of shape $\tau$ by $\mathcal{Y}(\tau) $
and let $V_{\tau}:=\operatorname{span} \{ T\colon T\in\mathcal{Y}(\tau)
 \} $ (the f\/ield is $\mathbb{C}(\kappa) $) with orthogonal
basis $\mathcal{Y}(\tau) $. Furthermore $\dim V_{\tau
}=\#\mathcal{Y}(\tau) =N!/\prod\limits_{(i,j)
\in\tau}\operatorname{hook} ( \tau;i,j ) $. For $1\leq i\leq N$ and
$T\in\mathcal{Y}(\tau) $ the entry~$i$ is at coordinates
$( \operatorname{rw}(i,T) ,\operatorname{cm}(i,T)) $ and the
\textit{content} is $c(i,T) :=\operatorname{cm}(i,T) -\operatorname{rw}(i,T) $. Each $T\in\mathcal{Y}(\tau) $ is uniquely
determined by its \textit{content vector} $[ c(i,T)
] _{i=1}^{N}$. For the example $\tau=(3,1) $
\begin{gather*}
\begin{bmatrix}
4 & 2 & 1\\
3 & &
\end{bmatrix}
,\qquad
\begin{bmatrix}
4 & 3 & 1\\
2 & &
\end{bmatrix}
,\qquad
\begin{bmatrix}
4 & 3 & 2\\
1 & &
\end{bmatrix}
\end{gather*}
the list of content vectors is $[ 2,1,-1,0]$, $[
2,-1,1,0]$, $[-1,2,1,0] $. To recover $T$ from its content
vector f\/ill in the entries starting with $N$, then $N-1$ ($c(
N-1,T) =\pm1$) has two possibilities and so on.

\begin{Example}
The list of $\mathcal{Y}(\tau) $ for $\tau=( 3,1,1)$, $N=5$
\begin{gather*}
\begin{bmatrix}
5 & 2 & 1\\
4 & & \\
3 & &
\end{bmatrix}
,\quad
\begin{bmatrix}
5 & 3 & 1\\
4 & & \\
2 & &
\end{bmatrix}
,\quad
\begin{bmatrix}
5 & 3 & 2\\
4 & & \\
1 & &
\end{bmatrix}
,\quad
\begin{bmatrix}
5 & 4 & 1\\
3 & & \\
2 & &
\end{bmatrix}
,\quad
\begin{bmatrix}
5 & 4 & 2\\
3 & & \\
1 & &
\end{bmatrix}
,\quad
\begin{bmatrix}
5 & 4 & 3\\
2 & & \\
1 & &
\end{bmatrix}.
\end{gather*}
The corresponding list of content vectors is $ [ 2,1,-2,-1,0 ] $,
$ [ 2,-2,1,-1,0 ] $, $ [ -2,2,1,-1,0 ] $, $ [
2,-2,-1,1,0 ] $, $ [ -2,2,-1,1,0 ] $, $ [
-2,-1,2,1,0 ] $.
\end{Example}

The representation theory can be developed using the content vectors in place
of tableaux; this is due to Okounkov and Vershik~\cite{Okounkov/Vershik2005}.

\subsection[Description of the representation $\tau$]{Description of the representation $\boldsymbol{\tau}$}\label{proptau}

The formulae for the action of $\tau(s_{i}) $ on the basis
$\mathcal{Y}(\tau) $ are from Murphy \cite[Theorem~3.12]{Murphy1981}. Def\/ine $b_{i}(T) :=1/( c(i,T)
-c( i+1,T) ) $. Note that $c(i,T) -c(
i+1,T) =0$ is impossible for RSYT's. If $\vert c(
i,T) -c( i+1,T) \vert \geq2$ let $T^{(
i) }\in\mathcal{Y}(\tau) $ denote $T$ with $i$, $i+1$
interchanged. The following describes the action of $\tau(s_{i})$ (in each case there is an informal subrectangle description of the relative
positions of $i$ and $i+1$ in~$T$; in cases~(3) and~(4) $i$ and $i+1$ are not
necessarily in adjacent rows or columns)
\begin{enumerate}\itemsep=0pt
\item \label{taustep1}If $\operatorname{rw}(i,T) =\operatorname{rw}(i+1,T) $
then $\tau(s_{i}) T=T$; position is $[ i+1,i] $,
$b_{i}(T) =1$.

\item \label{taustep2}If $\operatorname{cm}(i,T) =\operatorname{cm}(i+1,T) $
then $\tau(s_{i}) T=-T$; position is $
\begin{bmatrix}
i+1\\
i
\end{bmatrix}
$, $b_{i}(T) =-1$.

\item if $\operatorname{rw}(i,T) <\operatorname{rw}(i+1,T)$ (then $\operatorname{cm}(i,T)
>\operatorname{cm}(i+1,T) $), position $
\begin{bmatrix}
\ast & i\\
i+1 & \ast
\end{bmatrix}
$, $c(i,T) \geq ( \operatorname{cm}(i+1,T) +1 )
- ( \operatorname{rw}(i+1,T) -1 ) \geq c(i+1,T)
+2$, $0<b_{i}(T) \leq\frac{1}{2}$ then%
\begin{gather*}
\tau(s_{i}) T =T^{(i) }+b_{i}(T)
T,\\
\tau(s_{i}) T^{(i) } = \big( 1-b_{i} (
T ) ^{2} \big) T-b_{i}(T) T^{(i) }.
\end{gather*}

\item if $\operatorname{rw}(i,T) >\operatorname{rw}(i+1,T)$ (and $\operatorname{cm}(i,T)
<\operatorname{cm}(i+1,T) $), position $
\begin{bmatrix}
\ast & i+1\\
i & \ast
\end{bmatrix}
$; the formula is found in case (3) interchanging $T$ and $T^{(
i) }$, and using $b_{i}(T) =-b_{i}( T^{(
i) }) $.
\end{enumerate}

To eliminate extra parentheses we will write $\tau(i,j) $ for
$\tau((i,j)) $; where $(i,j) $ is a transposition.

There is a (unique up to constant multiple) positive Hermitian form on
$V_{\tau}$ for which $\tau$ is unitary (real orthogonal), that is
$\langle \tau(w) S_{1},S_{2}\rangle _{0}=\langle
S_{1},\tau(w) ^{-1}S_{2}\rangle _{0}\allowbreak
=\langle S_{1},\tau(w) ^{\ast}S_{2}\rangle _{0}$,
($S_{1},S_{2}\in V_{\tau}$, $w\in\mathcal{S}_{N}$):

\begin{Definition}
\begin{gather*}
\big\langle T,T^{\prime}\big\rangle _{0}:=\delta_{T,T^{\prime}}\times
\prod_{\substack{1\leq i<j\leq N,\\c(i,T) \leq c (
j,T ) -2}}\left( 1-\frac{1}{( c(i,T) -c (
j,T ) ) ^{2}}\right) ,\qquad T,T^{\prime}\in\mathcal{Y} (
\tau ).
\end{gather*}
\end{Definition}

The verif\/ication of the unitary property is based on the relation
\begin{gather*}
 \big\langle T^{(i) },T^{(i) } \big\rangle
_{0}=\big( 1-b_{i}(T) ^{2}\big) \langle
T,T \rangle _{0}
\end{gather*}
 when $0<b_{i}(T) \leq\frac{1}{2}$. Each
$\tau(w) $ is an orthogonal matrix with respect to the
orthonormal basis $ \big\{ \langle T,T \rangle _{0}^{-1/2}T\colon T\in\mathcal{Y}(\tau) \big\} $. The basis vectors $T$ are
simultaneous eigenvectors of the (reverse) \textit{Jucys--Murphy elements}
$\omega_{i}:=\sum\limits_{j=i+1}^{N}(i,j) $ (with $\omega_{N}=0$),
which commute pairwise and $\tau(\omega_{i}) T=c(i,T) T,$ for $1\leq i\leq N$ (see \cite[Lemma~3.6]{Murphy1981}); as
usual, $\tau$ is extended to a~homomorphism of the group algebra
$\mathbb{C}\mathcal{S}_{N}$ by $\tau\left( \sum_{w}b_{w}w\right) =\sum
_{w}b_{w}\tau(w) $.

\subsection{Vector-valued nonsymmetric Jack polynomials}

The main concern of this paper is $\mathcal{P}_{\tau}=\mathcal{P}\otimes
V_{\tau}$, the space of $V_{\tau}$ valued polynomials in $x$, which is
equipped with the $\mathcal{S}_{N}$ action:
\begin{gather*}
w\big( x^{\alpha}\otimes T\big) =(xw) ^{\alpha}\otimes
\tau(w) T,\qquad \alpha\in\mathbb{N}_{0}^{N}, \qquad T\in\mathcal{Y} (\tau) ,
\end{gather*}
extended by linearity to
\begin{gather*}
wp(x) =\tau(w) p(xw) ,\qquad p\in
\mathcal{P}_{\tau}.
\end{gather*}

\begin{Definition}
The \textit{Dunkl} and \textit{Cherednik--Dunkl} operators are ($1\leq i\leq
N$, $p\in\mathcal{P}_{\tau}$)
\begin{gather*}
\mathcal{D}_{i}p(x) :=\partial_{i}p(x)
+\kappa\sum_{j\neq i}\tau(i,j) \frac{p(x)
-p(x(i,j)) }{x_{i}-x_{j}},\\
\mathcal{U}_{i}p(x) :=\mathcal{D}_{i}( x_{i}p(
x)) -\kappa\sum_{j=1}^{i-1}\tau(i,j) p(x(i,j)) .
\end{gather*}
\end{Definition}

The commutation relations analogous to the scalar case hold, that is,
\begin{gather*}
\mathcal{D}_{i}\mathcal{D}_{j} =\mathcal{D}_{j}\mathcal{D}_{i},\qquad
\mathcal{U}_{i}\mathcal{U}_{j}=\mathcal{U}_{j}\mathcal{U}_{i}, \qquad 1\leq i,j\leq N,\\
w\mathcal{D}_{i} =\mathcal{D}_{w(i) }w, \qquad \forall\,
w\in\mathcal{S}_{N}, \qquad s_{j}\mathcal{U}_{i}=\mathcal{U}_{i}s_{j}, \qquad j\neq i-1,i,\\
s_{i}\mathcal{U}_{i}s_{i} =\mathcal{U}_{i+1}+\kappa s_{i}, \qquad \mathcal{U}
_{i}s_{i}=s_{i}\mathcal{U}_{i+1}+\kappa, \qquad \mathcal{U}_{i+1}s_{i}=s_{i}\mathcal{U}_{i}-\kappa.
\end{gather*}
The simultaneous eigenfunctions of $ \{\mathcal{U}_{i}\} $ are
called (vector-valued) nonsymmetric Jack polynomials (NSJP). For generic
$\kappa$ these eigenfunctions form a basis of $\mathcal{P}_{\tau}$ (we will
specify the excluded rational values in the sequel). They have a triangularity
property with respect to the partial order $\vartriangleright$. However the
structure does not merely rely on leading terms of the type $x^{\alpha}\otimes
T$. We need the rank function:

\begin{Definition}
For $\alpha\in\mathbb{N}_{0}^{N}$, $1\leq i\leq N$
\begin{gather*}
r_{\alpha}(i) :=\# \{ j\colon \alpha_{j}>\alpha_{i} \}
+\# \{ j\colon 1\leq j\leq i,\, \alpha_{j}=\alpha_{i} \} ,
\end{gather*}
then $r_{\alpha}\in\mathcal{S}_{N}.$
\end{Definition}

A consequence is that $r_{\alpha}\alpha=\alpha^{+}$, the nonincreasing
rearrangement of~$\alpha$, for any $\alpha\in\mathbb{N}_{0}^{N}$ . For example
if $\alpha= ( 1,2,1,4 ) $ then $r_{\alpha}= [ 3,2,4,1 ]
$ and $r_{\alpha}\alpha=\alpha^{+}= ( 4,2,1,1 ) $ (recall
$w\alpha_{i}=\alpha_{w^{-1}(i) }$). Also $r_{\alpha}=I$ if and
only if $\alpha$ is a partition ($\alpha_{1}\geq\alpha_{2}\geq\cdots\geq
\alpha_{N}$).

\looseness=-1
For each $\alpha\in\mathbb{N}_{0}^{N}$ and $T\in\mathcal{Y} (
\tau ) $ there is a NSJP $\zeta_{\alpha,T}$ with leading term
$x^{\alpha}\otimes\tau\big( r_{\alpha}^{-1}\big) T$, that is,
\begin{gather*}
\zeta_{\alpha,T} =x^{\alpha}\otimes\tau\big( r_{\alpha}^{-1}\big)
T+\sum_{\alpha\vartriangleright\beta}x^{\beta}\otimes t_{\alpha\beta} (
\kappa ) ,\qquad t_{\alpha\beta}(\kappa) \in V_{\tau},\\
\mathcal{U}_{i}\zeta_{\alpha,T} =\big( \alpha_{i}+1+\kappa c (
r_{\alpha}(i) ,T ) \big) \zeta_{\alpha,T}, \qquad 1\leq i\leq
N.
\end{gather*}

\subsection{The Yang--Baxter graph}

The NSJP's can be constructed by means of a Yang--Baxter graph. The details are
in~\cite{Dunkl/Luque2011}; this paper has several f\/igures illustrating some typical graphs.

A node consists of
\begin{gather*}
 ( \alpha,T,\xi_{\alpha.T},r_{\alpha},\zeta_{\alpha,T} ),
\end{gather*}
where $\alpha\in\mathbb{N}_{0}^{N}$, $\xi_{\alpha,T}$ is the spectral vector
$\xi_{\alpha,T}(i) =\alpha_{i}+1+\kappa c ( r_{\alpha
}(i) ,T ) $, $1\leq i\leq N$. The root is $\big(
\mathbf{0},T_{0}, [ 1+\kappa c ( i,T_{0} ) ] _{i=1}
^{N},I,1\otimes T_{0} \big) $ where $T_{0}$ is formed by entering
$N,N-1,\ldots,1$ column-by-column in the Ferrers diagram, for example
$\tau= ( 3,3,1 ) $
\begin{gather*}
T_{0}=
\begin{bmatrix}
7 & 4 & 2\\
6 & 3 & 1\\
5 & &
\end{bmatrix}
,\qquad c ( \cdot,T_{0}) =[ 1,2,0,1,-2,-1,0] .
\end{gather*}
There is an adjacency relation in $\mathcal{Y}(\tau) $ based on
the positions of the pairs $\ \{ i,i+1 \} $ and an inversion counter.

\begin{Definition}
For $T\in\mathcal{Y}(\tau) $ set
\begin{gather*}
\operatorname{inv}(T) :=\#\big\{ (i,j) \colon i<j,\, c (
i,T ) -c ( j,T ) \leq-2\big\} .
\end{gather*}
\end{Definition}

Recall from Section~\ref{proptau} that there are four types of positions of
a given pair $ \{ i,i+1 \} $ in~$T$, and in case~(3) it is
straightforward to check that $\operatorname{inv} ( T^{(i)
} ) =\operatorname{inv}(T) +1$.

If $\alpha_{i}\neq\alpha_{i+1}$ then $r_{s_{i}\alpha}=r_{\alpha}s_{i}$. The
cycle $w_{0}:= ( 123\ldots N ) $ and the af\/f\/ine transformation%
\begin{gather*}
\Phi ( a_{1},a_{2},\ldots,a_{N} ) := ( a_{2},a_{3},\ldots,a_{N},a_{1}+1 )
\end{gather*}
are fundamental parts of the construction; and $r_{\Phi\alpha}=r_{\alpha}w_{0}$ for any $\alpha$, that is,
\begin{gather*}
r_{\alpha}w_{0}(i) =r_{\alpha} ( w_{0}(i)
 ) =r_{\alpha}(i+1) =r_{\Phi\alpha}(i)
, \qquad 1\leq i<N,\\
r_{\alpha}w_{0}(N) =r_{\alpha} ( w_{0}(N)
 ) =r_{\alpha}(1) =r_{\Phi\alpha}(N) .
\end{gather*}

The \textit{jumps} in the graph, which raise the degree by one, are
\begin{gather}
 ( \alpha,T,\xi_{\alpha,T},r_{\alpha},\zeta_{\alpha,T} )
\overset{\Phi}{\longrightarrow} \big( \Phi\alpha,T,\Phi\xi_{\alpha
,T},r_{\alpha}w_{0},x_{N}w_{0}^{-1}\zeta_{\alpha,T}\big) ,\label{zjump}\\
\zeta_{\Phi\alpha,T}=x_{N}w_{0}^{-1}\zeta_{\alpha,T}\nonumber
\end{gather}
the leading term is $x^{\Phi\alpha}\otimes\tau\big( w_{0}^{-1}r_{\alpha
}^{-1}\big) T$ and $w_{0}^{-1}r_{\alpha}^{-1}= ( r_{\alpha}
w_{0} ) ^{-1}$. For example: $\alpha= (0,3,5,0)$,
$r_{\alpha}= [ 3,2,1,4 ]$, $\Phi\alpha= ( 3,5,0,1 )$,
$r_{\Phi\alpha}= [ 2,1,4,3 ] $.

There are two types of \textit{steps}, labeled by $s_{i}$:
\begin{enumerate} \itemsep=0pt
\item If $\alpha_{i}<\alpha_{i+1}$, then
\begin{gather*}
 ( \alpha,T,\xi_{\alpha,T},r_{\alpha},\zeta_{\alpha,T} )
\overset{s_{i}}{\longrightarrow} ( s_{i}\alpha,T,s_{i}\xi_{\alpha
,T},r_{\alpha}s_{i},\zeta_{s_{i}\alpha,T} ),
\\
\zeta_{s_{i}\alpha,T}=s_{i}\zeta_{\alpha,T}-\frac{\kappa}{\xi_{\alpha
,T}(i) -\xi_{\alpha,T}(i+1) }\zeta_{\alpha,T}.
\end{gather*}
Observe that this construction is valid provided $\xi_{\alpha,T} (
i ) \neq\xi_{\alpha,T}(i+1) $, that is, $\alpha
_{i+1}-\alpha_{i}\neq\kappa ( c ( r_{\alpha}(i)
,T ) -c ( r_{\alpha}(i+1) ,T ) ) $. The
extreme values of $c ( \cdot,T ) $ are $\tau_{1}-1$ and
$1-\ell(\tau) $, thus $ \vert c ( r_{\alpha} (
i ) ,T ) -c ( r_{\alpha+1}(i) ,T )
 \vert \allowbreak\leq h_{\tau}-1$. Furthermore $\alpha_{i+1}-\alpha
_{i}\geq1$ and the step is valid provided $\kappa m\notin \{
1,2,3,\dots \} $ for $m=1-h_{\tau},2-h_{\tau},\ldots,h_{\tau}-1$. The
bound $-1/( h_{\tau}-1) <\kappa<1/( h_{\tau}-1)$
is suf\/f\/icient.

\item If $\alpha_{i}=\alpha_{i+1}$, and the positions of $j:=r_{\alpha}(
i )$, $j+1$ in $T$ are of type (3), that is, $c(j,T)
-c(j+1,T) \geq2$ (the def\/inition of $r_{\alpha}$ implies
$r_{\alpha}(i+1) =j+1$ and $s_{i}r_{\alpha}^{-1}=r_{\alpha}%
^{-1}s_{j}$). Set
\begin{gather*}
b^{\prime}=\frac{1}{c(j,T) -c(j+1,T) }
=\frac{\kappa}{\xi_{\alpha,T}(i) -\xi_{\alpha,T}(
i+1) };
\end{gather*}
thus $0<b^{\prime}\leq\frac{1}{2}$, there is a step%
\begin{gather*}
 ( \alpha,T,\xi_{\alpha,T},r_{\alpha},\zeta_{\alpha,T} )
\overset{s_{i}}{\longrightarrow} \big( \alpha,T^{(j) },s_{i}
\xi_{\alpha,T},r_{\alpha},\zeta_{\alpha,T^{(j) }}\big) ,
\\
\zeta_{\alpha,T^{(j) }}=s_{i}\zeta_{\alpha,T}-b^{\prime}
\zeta_{\alpha,T},
\end{gather*}
($T^{(j) }$ is the result of interchanging $j$ and $j+1$ in
$T$). The leading term is transformed $s_{i}\big( x^{\alpha}\otimes
\tau\big( r_{\alpha}^{-1}\big) T\big) = ( xs_{i} )
^{\alpha}\otimes\tau\big( s_{i}r_{\alpha}^{-1}\big) T\allowbreak
=x^{\alpha}\otimes\tau\big( r_{\alpha}^{-1}\big) \tau (
s_{j} ) T$ and $\tau ( s_{j} ) T=T^{(j)}+b^{\prime}T$.
\end{enumerate}

There are two other possibilities corresponding to~(\ref{taustep1}) and~(\ref{taustep2}) for the action of $s_{i}$ on $\zeta_{\alpha,T}$ when
$\alpha_{i}=\alpha_{i+1}$ (note $r_{\alpha}(i+1) =r_{\alpha
}(i) +1$): (1) $\operatorname{rw}( r_{\alpha}(i)
,T) =\operatorname{rw}( r_{\alpha}(i) +1,T) $,
then $s_{i}\zeta_{\alpha,T}=\zeta_{\alpha,T}$; (2) $\operatorname{cm}(
r_{\alpha}(i) ,T) =\operatorname{cm}( r_{\alpha}(
i) +1,T) $, then $s_{i}\zeta_{\alpha,T}=-\zeta_{\alpha,T}$.

The proofs that these formulae are mutually compatible for dif\/ferent paths in
the graph from the root $( \mathbf{0},T_{0}) $ to a given node
$(\alpha,T) $, use inductive arguments based on the fact that
these paths have the same length. The number of jumps is clearly $\vert
\alpha\vert $ and the number of steps is $S(\alpha)
+\operatorname{inv}(T) -\operatorname{inv}(T_{0}) $, where
\begin{gather*}
S(\alpha) :=\frac{1}{2}\sum_{1\leq i<j\leq N} (\vert
\alpha_{i}-\alpha_{j} \vert + \vert \alpha_{i}-\alpha_{j}
+1 \vert -1 ) .
\end{gather*}

\section{Hermitian forms}\label{section3}

For a complex vector space $V$ a Hermitian form is a mapping $ \langle
\cdot,\cdot \rangle\colon V\otimes V\rightarrow\mathbb{C}$ such that
$\langle u,cv\rangle =c \langle u,v \rangle $,
$ \langle u,v_{1}+v_{2} \rangle = \langle u,v_{1} \rangle
+ \langle u,v_{2} \rangle $ and $ \langle u,v \rangle
=\overline{ \langle v,u \rangle }$ for $u,v_{1},v_{2}\in V$,
$c\in\mathbb{C}$. The form is positive semidef\/inite if $ \langle
u,u \rangle \geq0$ for all $u\in V$. The concern of this paper is with a~particular Hermitian form on $\mathcal{P}_{\tau}$ which has the properties
(for all $f,g\in\mathcal{P}_{\tau}$):
\begin{gather}
\big\langle 1\otimes T,1\otimes T^{\prime}\big\rangle =\big\langle
T,T^{\prime}\big\rangle _{0}, \qquad T,T^{\prime}\in\mathcal{Y}(\tau),\label{admforms}\\
 \langle wf,wg \rangle = \langle f,g \rangle, \qquad w\in\mathcal{S}_{N},\nonumber\\
 \langle x_{i}\mathcal{D}_{i}f,g \rangle = \langle
f,x_{i}\mathcal{D}_{i}g \rangle , \qquad 1\leq i\leq N.\nonumber
\end{gather}
The commutation $\mathcal{U}_{i}=\mathcal{D}_{i}x_{i}-\kappa\sum\limits_{j<i}(
i,j) =x_{i}\mathcal{D}_{i}+1+\kappa\sum\limits_{j>i}(i,j) $
together with $ \langle (i,j) f,g \rangle = \langle
f,(i,j) g \rangle $ show that $ \langle \mathcal{U}
_{i}f,g \rangle = \langle f,\mathcal{U}_{i}g \rangle $ for all~$i$. Thus the uniqueness of the spectral vectors discussed above implies that
$ \langle \zeta_{\alpha,T},\zeta_{\beta,T^{\prime}} \rangle =0$
whenever $(\alpha,T) \neq ( \beta,T^{\prime} ) $. In
particular polynomials homogeneous of dif\/ferent degrees are mutually
orthogonal, by the basis property of $ \{ \zeta_{\alpha,T} \} $. We
can deduce contiguity relations corresponding to the steps described above and
implied by the properties of the form. Consider step type~(1) with
\begin{gather*}
s_{i}\zeta_{\alpha,T} =\zeta_{s_{i}\alpha,T}+b^{\prime}\zeta_{\alpha,T},\qquad
b^{\prime} =\frac{\kappa}{\xi_{\alpha,T}(i) -\xi_{\alpha
,T}(i+1) }.
\end{gather*}
The conditions $ \langle s_{i}\zeta_{\alpha,T},s_{i}\zeta_{\alpha
,T} \rangle = \langle \zeta_{\alpha,T},\zeta_{\alpha,T}\rangle
$ and $\langle \zeta_{\alpha,T},\zeta_{s_{i}\alpha,T}\rangle =0$
imply
\begin{gather*}
 \langle \zeta_{\alpha,T},\zeta_{\alpha,T} \rangle =\big\langle
\zeta_{s_{i}\alpha,T}+b^{\prime}\zeta_{\alpha,T},\zeta_{s_{i}\alpha
,T}+b^{\prime}\zeta_{\alpha,T}\big\rangle
 = \langle \zeta_{s_{i}\alpha,T},\zeta_{s_{i}\alpha,T} \rangle
+b^{\prime2} \langle \zeta_{\alpha,T},\zeta_{\alpha,T} \rangle ,\\
 \langle \zeta_{s_{i}\alpha,T},\zeta_{s_{i}\alpha,T} \rangle
=\big( 1-b^{\prime2}\big) \langle \zeta_{\alpha,T},\zeta_{\alpha
,T} \rangle .
\end{gather*}
A necessary condition that the form be positive-def\/inite ($f\neq0$ implies
$\langle f,f\rangle >0$) is that $-1<b^{\prime}<1$ in each of the
possible steps. Since (with $j=r_{\alpha}(i) $ and
$\ell=r_{\alpha}(i+1) $)
\begin{gather*}
1-b^{\prime2}=
\frac{[ \alpha_{i+1}-\alpha_{i}+( c( \ell,T)
-c(j,T) +1) \kappa] [ \alpha_{i+1}-\alpha_{i}+( c( \ell,T) -c(j,T) -1)
\kappa] }{[ \alpha_{i+1}-\alpha_{i}+( c(
\ell,T) -c(j,T) ) \kappa] ^{2}},
\end{gather*}
the extreme values of $( c(\ell,T) -c(j,T)
\pm1) $ are $\pm h_{\tau}$, and $\alpha_{i+1}-\alpha_{i}\geq1$, it
follows that $-1/h_{\tau}<\kappa<1/h_{\tau}$ implies $1-b^{\prime2}>0$. Since
steps of type (1) link any $(\alpha,T) $ to $( \alpha
^{+},T) $ one can obtain (with $\varepsilon=\pm1$)
\begin{gather}
\mathcal{E}_{\varepsilon}(\alpha,T) :=\prod_{\substack{1\leq
i<j\leq N\\ \alpha_{i}<\alpha_{j}}}\left( 1+\frac{\varepsilon\kappa}%
{\alpha_{j}-\alpha_{i}+\kappa ( c ( r_{\alpha}(j)
,T ) -c ( r_{\alpha}(i) ,T ) ) }\right),
\label{EEform}\\
\langle \zeta_{\alpha,T},\zeta_{\alpha,T}\rangle =\big(
\mathcal{E}_{1}(\alpha,T) \mathcal{E}_{-1} ( \alpha
,T ) \big) ^{-1} \langle \zeta_{\alpha^{+},T},\zeta_{\alpha
^{+},T} \rangle .\nonumber
\end{gather}

Similarly the steps of type (2) (with $\alpha_{i}=\alpha_{i+1}$ and
$j=r_{\alpha}(i) $, $b^{\prime}=\frac{1}{c(j,T)
-c(j+1,T) }$) imply the relation%
\begin{gather*}
\big\langle \zeta_{\alpha,T^{(j) }},\zeta_{\alpha,T^{(j) }}\big\rangle =\big( 1-b^{\prime2}\big) \langle
\zeta_{\alpha,T},\zeta_{\alpha,T} \rangle .
\end{gather*}

It was shown in \cite{Dunkl2010} (this is a special case of a result of
Grif\/feth \cite[Theorem~6.1]{Griffeth2010}) that the def\/inition for $\lambda
\in\mathbb{N}_{0}^{N,+}$
\begin{gather*}
 \langle \zeta_{\lambda,T},\zeta_{\lambda,T} \rangle
= \langle T,T \rangle _{0}\prod_{i=1}^{N} ( 1+\kappa c (
i,T ) ) _{\lambda_{i}}
 \prod_{1\leq i<j\leq N}\prod_{\ell=1}^{\lambda_{i}-\lambda_{j}
} \left( 1-\left( \frac{\kappa}{\ell+\kappa ( c(i,T)
-c(j,T) ) }\right) ^{2}\right) .
\end{gather*}
together with formula (\ref{EEform}) produce a Hermitian form (called the
covariant form) satis\-fying~(\ref{admforms}) and the additional property
$\langle x_{i}f,g\rangle =\langle f,\mathcal{D}_{i}%
g\rangle $ for all $f,g\in\mathcal{P}_{\tau}$ and $1\leq i\leq N$ (the
bound $-1/h_{\tau}<\kappa<1/h_{\tau}$ for positivity of this form was found by
Etingof and Stoica~\cite{Etingof/Stoica2009}).

Here we want a Hermitian form for which multiplication by any $x_{i}$ is an
isometry, that is, $ \langle x_{i}f,x_{i}g \rangle = \langle
f,g \rangle $ for all $f,g\in\mathcal{P}_{\tau}$ and $1\leq i\leq N$.
Heuristically this should involve an integral over the $N$-torus. The isometry
postulate, and the equations~(\ref{admforms}) determine the form uniquely, as
will be shown. The postulate $\langle x_{N}f,g\rangle =\langle
f,\mathcal{D}_{N}g\rangle $ in the covariant form is used to compute the
ef\/fect of a jump $\zeta_{\Phi\alpha,T}=x_{N}w_{0}^{-1}\zeta_{\alpha,T}$ (see~(\ref{zjump})), that is, to evaluate $\langle \zeta_{\Phi\alpha,T}
,\zeta_{\Phi\alpha,T}\rangle /\langle \zeta_{\alpha,T}
,\zeta_{\alpha,T}\rangle $. From the proofs in \cite[Appendix, Corollary~5, Theorem~10]{Dunkl2010} and~\cite{Dunkl/Luque2011} we see that the factor $\prod\limits_{i=1}^{N} ( 1+\kappa c(i,T) ) _{\lambda_{i}}$ arises from
ratios of this type. This aspect (here we need the ratio to be~$1$) motivates
the following:

\begin{Definition}\label{Tformdef}For $\lambda\in\mathbb{N}_{0}^{N,+}$, $\alpha,\beta
\in\mathbb{N}_{0}^{N}$, and $T,T^{\prime}\in\mathcal{Y}(\tau) $
the Hermitian form $\langle \cdot,\cdot\rangle _{\mathbb{T}}$ on
$\mathcal{P}_{\tau}$ is specif\/ied by
\begin{gather*}
(\alpha,T) \neq ( \beta,T^{\prime} )
\Longrightarrow \langle \zeta_{\alpha,T},\zeta_{\beta,T} \rangle
_{\mathbb{T}}=0,\\
 \langle \zeta_{\lambda,T},\zeta_{\lambda,T} \rangle _{\mathbb{T}}
 = \langle T,T \rangle _{0}\prod_{1\leq i<j\leq N}\prod_{\ell
=1}^{\lambda_{i}-\lambda_{j}}\left( 1-\left( \frac{\kappa}{\ell
+\kappa ( c(i,T) -c(j,T) ) }\right)
^{2}\right) ,\\
 \langle \zeta_{\alpha,T},\zeta_{\alpha,T} \rangle _{\mathbb{T}}
=\big( \mathcal{E}_{1}(\alpha,T) \mathcal{E}_{-1} (
\alpha,T ) \big) ^{-1} \langle \zeta_{\alpha^{+},T},\zeta
_{\alpha^{+},T} \rangle _{\mathbb{T}};
\end{gather*}
the form is extended to all of $\mathcal{P}_{\tau}$ by linearity in the second
variable and Hermitian symmetry, that is, $ \langle f,c_{1}g+c_{2}
h \rangle _{\mathbb{T}}=c_{1} \langle f,g \rangle _{\mathbb{T}
}+c_{2} \langle f,h \rangle _{\mathbb{T}}$ and $\langle
f,g \rangle _{\mathbb{T}}=\overline{\langle g,f\rangle
_{\mathbb{T}}}$, for $f,g,h\in\mathcal{P}_{\tau}$ and $c_{1},c_{2}\in\mathbb{C}$.
\end{Definition}

Observe that the formula is invariant when $\lambda$ is replaced by
$\lambda+m\boldsymbol{1=} ( \lambda_{1}+m,\lambda_{2}+m,\ldots
$, $\lambda_{N}+m )$ for any $m\in\mathbb{N}$. This follows easily from
the commutation (where $e_{N}:=x_{1}x_{2}\cdots x_{N}$)
\begin{gather*}
\mathcal{U}_{i}\big( e_{N}^{m}f\big) =me_{N}^{m}f+e_{N}^{m}
\mathcal{U}_{i}f, \qquad 1\leq i\leq N, \qquad m=1,2,\ldots,
\end{gather*}
thus $\mathcal{U}_{i} ( e_{N}^{m}\zeta_{\alpha,T} ) = (
m+\alpha_{i}+\kappa c ( r_{\alpha}(i) ,T ) )
e_{N}^{m}\zeta_{\alpha,T}$, and $e_{N}^{m}\zeta_{\alpha,T}$ is a simultaneous
eigenfunction of~$ \{ \mathcal{U}_{i} \} $ with the same
eigenvalues and the same leading term as $\zeta_{\alpha+m\boldsymbol{1},T}$.
Hence $\zeta_{\alpha+m\boldsymbol{1},T}=e_{N}^{m}\zeta_{\alpha,T}$. We now
extend the structure of NSJP's to $V_{\tau}$-valued Laurent polynomials,
thereby producing a~basis:

\begin{Definition}
Suppose $\alpha\in\mathbb{Z}^{N}$ then set $\zeta_{\alpha,T}=e_{N}^{-m}
\zeta_{\alpha+m\boldsymbol{1},T}$ where $m\in\mathbb{N}_{0}$ and satisf\/ies
$m\geq-\min_{j}\alpha_{i}$. This is valid since $\alpha+m\boldsymbol{1}
\in\mathbb{N}_{0}^{N}$ and by the relation $\zeta_{\beta+k\boldsymbol{1}
,T}=e_{N}^{k}\zeta_{\beta,T}$ for $\beta\in\mathbb{N}_{0}^{N}$ and
$k\in\mathbb{N}_{0}$.
\end{Definition}

The proof that the form satisf\/ies the properties (\ref{admforms}) with respect
to steps is the same as the one in \cite[Propositions~8 and~9]{Dunkl2010}, and it
suf\/f\/ices to verify the ef\/fect of a~jump.

\begin{Theorem}\label{norm4jump} Suppose $\alpha\in\mathbb{N}_{0}^{N,+}$ then $ \langle
\zeta_{\Phi\alpha,T},\zeta_{\Phi\alpha,T} \rangle _{\mathbb{T}
}= \langle \zeta_{\alpha,T},\zeta_{\alpha,T} \rangle _{\mathbb{T}}$.
\end{Theorem}

\begin{proof}
We use (\ref{zjump}) to relate $ \langle \zeta_{\Phi\alpha,T},\zeta
_{\Phi\alpha,T} \rangle _{\mathbb{T}}$ to $ \langle \zeta_{\beta
,T},\zeta_{\beta,T} \rangle _{\mathbb{T}}$ where $\beta= (
\Phi\alpha ) ^{+}= ( \alpha_{1}+1,\alpha_{2}$, $\ldots,\alpha
_{N} )$. The product $\mathcal{E}_{\varepsilon} ( \Phi
\alpha,T ) $ is over the pairs $\alpha_{i}= ( \Phi\alpha )
_{i-1}< ( \Phi\alpha ) _{N}=\alpha_{1}+1$ for $2\leq i\leq N$, thus
\begin{gather*}
\mathcal{E}_{\varepsilon} ( \Phi\alpha,T ) =\prod
\limits_{i=2}^{N}\left( 1-\frac{\varepsilon\kappa}{\alpha_{1}+1-\alpha
_{i}+\kappa ( c ( r_{\Phi\alpha}(N) ,T )
-c ( r_{\Phi\alpha} ( i-1 ) ,T ) ) } \right) \\
\hphantom{\mathcal{E}_{\varepsilon} ( \Phi\alpha,T )}{}
 =\prod\limits_{i=2}^{N}\left( 1-\frac{\varepsilon\kappa}{\alpha
_{1}+1-\alpha_{i}+\kappa ( c ( 1,T ) -c(i,T)
 ) }\right) .
\end{gather*}
By def\/inition
\begin{gather*}
 \langle \zeta_{\Phi\alpha,T},\zeta_{\Phi\alpha,T} \rangle
_{\mathbb{T}} =\big( \mathcal{E}_{1} ( \Phi a,T )
\mathcal{E}_{-1} ( \Phi\alpha,T ) \big) ^{-1} \langle
\zeta_{\beta,T},\zeta_{\beta,T} \rangle _{\mathbb{T}}\\
\hphantom{\langle \zeta_{\Phi\alpha,T},\zeta_{\Phi\alpha,T} \rangle _{\mathbb{T}} }{}
 =\prod\limits_{i=2}^{N}\left( 1-\left( \frac{\kappa}{\alpha_{1}
+1-\alpha_{i}+\kappa ( c ( 1,T ) -c(i,T)
 ) } \right) ^{2}\right) ^{-1}\\
\hphantom{\langle \zeta_{\Phi\alpha,T},\zeta_{\Phi\alpha,T} \rangle _{\mathbb{T}} =}{}
 \times \langle T,T \rangle _{0}\prod_{1\leq i<j\leq N}\prod
_{\ell=1}^{\beta_{i}-\beta_{j}}\left( 1-\left( \frac{\kappa}{\ell
+\kappa ( c(i,T) -c(j,T)) }\right)
^{2}\right) \\
\hphantom{\langle \zeta_{\Phi\alpha,T},\zeta_{\Phi\alpha,T} \rangle _{\mathbb{T}} }{}
 = \langle T,T \rangle _{0}\prod_{2\leq i<j\leq N}\prod_{\ell
=1}^{\alpha_{i}-\alpha_{j}}\left( 1-\left( \frac{\kappa}{\ell+\kappa (
c(i,T) -c(j,T) ) }\right) ^{2}\right) \\
\hphantom{\langle \zeta_{\Phi\alpha,T},\zeta_{\Phi\alpha,T} \rangle _{\mathbb{T}} =}{}
 \times\prod\limits_{j=2}^{N}\prod_{\ell=1}^{\alpha_{1}-\alpha_{j}}\left(
1-\left( \frac{\kappa}{\ell+\kappa( c( 1,T) -c(
j,T)) }\right) ^{2}\right)
 =\langle \zeta_{\alpha,T},\zeta_{\alpha,T}\rangle _{\mathbb{T}}.
\end{gather*}
The terms in the product for $i=1$ and $2\leq j\leq N$, $\ell=\beta_{1}
-\beta_{j}=\alpha_{1}+1-\alpha_{j}$ are canceled out.
\end{proof}

We summarize the key results. We say $\kappa$ is \textit{generic} if $(
\alpha,T) \neq( \beta,T^{\prime}) $ implies the spectral
vectors $\xi_{\alpha,T}\neq\xi_{\beta,T^{\prime}}$.

\begin{Proposition}
\label{torusform}For generic $\kappa$ the Hermitian form $ \langle
\cdot,\cdot \rangle _{\mathbb{T}}$ satisfies
\begin{enumerate}\itemsep=0pt
\item[$1)$] if $f$, $g$ are
homogeneous and $\deg f\neq\deg g$ then $\langle f,g\rangle
_{\mathbb{T}}=0$,
\item[$2)$] $\langle wf,wg\rangle _{\mathbb{T}
}=\langle f,g\rangle _{\mathbb{T}}$, $f,g\in\mathcal{P}_{\tau}$, $w\in\mathcal{S}_{N}$,
\item[$3)$] $ \langle x_{i}\mathcal{D}_{i}f,g \rangle _{\mathbb{T}}= \langle f,x_{i}\mathcal{D}_{i}
g \rangle _{\mathbb{T}}$ for $f,g\in\mathcal{P}_{\tau}$ and $1\leq i\leq
N$,
\item[$4)$] $\langle x_{i}f,x_{i}g\rangle _{\mathbb{T}}= \langle f,g \rangle _{\mathbb{T}}$ for $f,g\in\mathcal{P}_{\tau}$
and $1\leq i\leq N$.
\end{enumerate}
\end{Proposition}

\begin{proof}
For generic $\kappa$ the NSJP's $\zeta_{\alpha,T}$ with $\vert
\alpha \vert =n$ form a basis for $\mathcal{P}_{\tau,n}$; this
immediately implies~(1). For (2) the fact that $ \langle s_{i}
\zeta_{\alpha,T},s_{i}\zeta_{\beta,T^{\prime}} \rangle _{\mathbb{T}
}=\langle \zeta_{\alpha,T},\zeta_{\beta,T^{\prime}}\rangle
_{\mathbb{T}}$ for $1\leq i<N$ follows from the corresponding results in
\cite[Propositions~8 and~9, Corollary~3]{Dunkl2010} when $\vert\alpha\vert
= \vert \beta \vert $, otherwise from Def\/inition~\ref{Tformdef}. This
suf\/f\/ices for~(2) since $ \{ s_{i} \} $ generates~$\mathcal{S}_{N}$.
The def\/inition of NSJP's implies trivially that $\langle \mathcal{U}
_{i}\zeta_{\alpha,T},\zeta_{\beta,T^{\prime}}\rangle _{\mathbb{T}
}=\langle \zeta_{\alpha,T},\mathcal{U}_{i}\zeta_{\beta,T^{\prime}
}\rangle _{\mathbb{T}}$ for all $i$ and $(\alpha,T)$, $( \beta,T^{\prime}) $ because both sides vanish if $(
\alpha,T) \neq( \beta,T^{\prime}) $, otherwise equal
$\xi_{\alpha,T}(i) \langle \zeta_{\alpha,T},\zeta
_{\alpha,T}\rangle _{\mathbb{T}}$. The commutation $\mathcal{U}
_{i}=\mathcal{D}_{i}x_{i}-\kappa\sum\limits_{j<i}(i,j) =x_{i}
\mathcal{D}_{i}+1+\kappa\sum\limits_{j>i}(i,j) $ together with
$ \langle (i,j) f,g \rangle _{\mathbb{T}}= \langle
f,(i,j) g \rangle _{\mathbb{T}}$ from~(2) show that
$ \langle x_{i}\mathcal{D}_{i}f,g \rangle _{\mathbb{T}}= \langle
f,x_{i}\mathcal{D}_{i}g \rangle _{\mathbb{T}}$ for $1\leq i\leq N$. For
part (4) (recall $w_{0}= ( 123\ldots N ) $) $\zeta_{\Phi\alpha
,T}=x_{N}w_{0}^{-1}\zeta_{\alpha,T}$ and $\zeta_{\Phi\beta,T^{\prime}}
=x_{N}w_{0}^{-1}\zeta_{\beta,T^{\prime}}$. By Theorem~\ref{norm4jump}
$\langle \zeta_{\Phi\alpha,T},\zeta_{\Phi\beta,T^{\prime}}\rangle
_{\mathbb{T}}=\langle \zeta_{\alpha,T},\zeta_{\beta,T^{\prime}
}\rangle _{\mathbb{T}}$ (if $(\alpha,T) \neq(
\beta,T^{\prime}) $ then $( \Phi\alpha,T) \neq(
\Phi\beta,T^{\prime}) $). Thus for each $(\alpha,T)$, $(\beta,T^{\prime})$
\begin{gather*}
\big\langle x_{N}\big( w_{0}^{-1}\zeta_{\alpha,T}\big) ,x_{N}\big(
w_{0}^{-1}\zeta_{\beta,T^{\prime}}\big) \big\rangle _{\mathbb{T}
}= \langle \zeta_{\alpha,T},\zeta_{\beta,T^{\prime}} \rangle
_{\mathbb{T}}= \big\langle w_{0}^{-1}\zeta_{\alpha,T},w_{0}^{-1}\zeta
_{\beta,T^{\prime}} \big\rangle _{\mathbb{T}}.
\end{gather*}
The set $\big\{ w_{0}^{-1}\zeta_{\alpha,T}:(\alpha,T)
\big\} $ is a basis for $\mathcal{P}_{\tau}$ thus $\langle
x_{N}f,x_{N}g\rangle _{\mathbb{T}}=\langle f,g\rangle
_{\mathbb{T}}$ for all $f,g\in\mathcal{P}_{\tau}$. For any $i$
\begin{align*}
 \langle x_{i}f,x_{i}g \rangle _{\mathbb{T}} & = \langle
 ( i,N ) x_{i}f,(i,N) x_{i}g \rangle
_{\mathbb{T}}= \langle x_{N}(i,N) f,x_{N}(i,N)
g \rangle _{\mathbb{T}}\\
& = \langle (i,N) f,(i,N) g \rangle
_{\mathbb{T}}= \langle f,g \rangle _{\mathbb{T}};
\end{align*}
and this completes the proof.
\end{proof}

This lays the abstract foundation for the next developments.

\section{Fourier--Stieltjes coef\/f\/icients on the torus}\label{section4}

The torus $\mathbb{T}^{N}:=\big\{ x\in\mathbb{C}^{N}\colon \vert
x_{i}\vert =1,\, 1\leq i\leq N\big\} $ is a multiplicative compact
abelian group with dual group $\mathbb{Z}^{N}$. We will use this property to
f\/ind the measure of orthogonality for the NSJP's on the torus. First we
produce the Fourier--Stieltjes coef\/f\/icients of the hypothetical measure and
then use a matrix version of a theorem of Bochner to deduce the existence of
the measure.

When $\kappa$ is generic the NSJP's form a basis for $\mathcal{P}_{\tau}$ and
it is possible to make the def\/inition
\begin{gather*}
\widetilde{A}\big( \alpha,\beta,T,T^{\prime}\big) :=\big( \langle
T,T \rangle _{0}\big\langle T^{\prime},T^{\prime}\big\rangle
_{0}\big) ^{-1/2}\big\langle x^{\alpha}\otimes T,x^{\beta}\otimes
T^{\prime}\big\rangle _{\mathbb{T}}
\end{gather*}
for $\alpha,\beta\in\mathbb{N}_{0}^{N}$ and $T,T^{\prime}\in\mathcal{Y} (
\tau) $. In ef\/fect this uses the orthonormal basis of $V_{\tau}$. By
the symmetry of the form $\widetilde{A} ( \alpha,\beta,T,T^{\prime
} ) =\widetilde{A} ( \beta,\alpha,T^{\prime},T)$. By
Proposition~\ref{torusform} $\vert \alpha\vert \neq\vert
\beta\vert $ implies $\widetilde{A}( \alpha,\beta,T,T^{\prime
}) =0$. Another consequence is $\widetilde{A} (\mathbf{0}
,\mathbf{0},T,T^{\prime}) =\delta_{T,T^{\prime}}$.

\begin{Definition}
For each $\gamma\in\mathbb{Z}^{N}$ with $\sum_{i=1}^{N}\gamma_{i}=0$ let
$\gamma_{i}^{\pi}=\max( \gamma_{i},0) $ and $\gamma_{i}^{\nu
}=-\min( \gamma_{i},0) $ for $1\leq i\leq N$; then $\gamma
=\gamma^{\pi}-\gamma^{\nu}$ and $\gamma^{\pi},\gamma^{\nu}\in\mathbb{N}
_{0}^{N}$. Furthermore $\vert \gamma^{\pi}\vert =\vert
\gamma^{\nu}\vert $ and $\sum_{i}\vert \gamma_{i}\vert
=\vert \gamma^{\pi}\vert +\vert \gamma^{\nu}\vert $ is even.
\end{Definition}

Introduce the index set $\boldsymbol{Z}_{N}$ and its graded components by
\begin{gather*}
\boldsymbol{Z}_{N} :=\left\{ \alpha\in\mathbb{Z}^{N}\colon \sum_{i=1}^{N} \alpha_{i}=0\right\}, \\
\boldsymbol{Z}_{N,n} :=\left\{ \alpha\in\boldsymbol{Z}_{N}\colon \sum_{i=1}^{N} \vert \alpha_{i} \vert =2n\right\} ,\qquad n=0,1,2,\ldots.
\end{gather*}
Formula~\eqref{ZNnct} for $\#\boldsymbol{Z}_{N,n}$ is in Appendix~\ref{appendixA}.

\begin{Definition}
For $\gamma\in\mathbb{Z}^{N}$ the matrix $A_{\gamma}$ (of size $\#\mathcal{Y}
(\tau) \times\#\mathcal{Y}(\tau) $) is given by
\begin{gather*}
 ( A_{\gamma}) _{T,T^{\prime}} =\widetilde{A}\big(
\gamma^{\pi},\gamma^{\nu},T,T^{\prime}\big) , \qquad T,T^{\prime}\in
\mathcal{Y}(\tau) , \qquad \gamma\in\boldsymbol{Z}_{N},\\
A_{\gamma} =0, \qquad \gamma\notin\boldsymbol{Z}_{N}.
\end{gather*}
\end{Definition}

\begin{Proposition}
Suppose $\alpha,\beta\in\mathbb{N}_{0}^{N}$ and $T,T^{\prime}\in
\mathcal{Y}(\tau) $ then $\widetilde{A} ( \alpha
,\beta,T,T^{\prime} ) = ( A_{\alpha-\beta} ) _{T,T^{\prime}}$.
\end{Proposition}

\begin{proof}
If $\vert \alpha\vert \neq\vert \beta\vert $ then
$\sum\limits_{i=1}^{N} ( \alpha_{i}-\beta_{i} ) \neq0$, $\widetilde
{A} ( \alpha,\beta,T,T^{\prime} ) =0$ and $A_{\alpha-\beta}=0$ by
def\/inition. If $\vert \alpha\vert =\vert \beta\vert $
let $\zeta_{i}=\min( \alpha_{i},\beta_{i}) $ for $1\leq i\leq N$
then $x^{\zeta}$ is a factor of both $x^{\alpha}$ and $x^{\beta}$; by
Proposition~\ref{torusform} $\langle x^{\alpha}\otimes T,x^{\beta}\otimes
T^{\prime}\rangle _{\mathbb{T}}=\langle x^{\alpha-\zeta}\otimes
T,x^{\beta-\zeta}\otimes T^{\prime}\rangle _{\mathbb{T}}$ . By
construction $( \alpha-\beta) ^{\pi}=\alpha-\zeta,(
\alpha-\beta) ^{\nu}=\beta-\zeta$. It follows that
\begin{gather*}
\big\langle x^{\alpha}\otimes T,x^{\beta}\otimes T^{\prime}\big\rangle
=T^{\ast}A_{\alpha-\beta}T^{\prime}=\big( \langle T,T \rangle
_{0}\big\langle T^{\prime},T^{\prime}\big\rangle _{0}\big) ^{1/2} (
A_{\alpha-\beta} ) _{T,T^{\prime}}.\tag*{\qed}
\end{gather*}
\renewcommand{\qed}{}
\end{proof}

For a formal Laurent series $h(x) =\sum\limits_{\alpha
\in\mathbb{Z}^{N}}c_{\alpha}x^{\alpha}$ let $\operatorname{CT}( h(x)
) =c_{\mathbf{0}}$, the constant term. Then
\begin{gather*}
\big\langle x^{\alpha}\otimes T,x^{\beta}\otimes T^{\prime}\big\rangle
=\big( \langle T,T \rangle _{0}\big\langle T^{\prime},T^{\prime
}\big\rangle _{0}\big) ^{1/2} \operatorname{CT}\bigg( x^{-\alpha}\sum_{\gamma
\in\mathbb{Z}^{N}} ( A_{\gamma} ) _{T,T^{\prime}}x^{\gamma}%
x^{\beta}\bigg) .
\end{gather*}
In the next section we investigate analytical properties of the formal series,
but f\/irst we consider algebraic properties, that is, those not needing any
convergence results.

\begin{Theorem}
Suppose $\gamma\in\boldsymbol{Z}_{N}$ and $w\in\mathcal{S}_{N}$ then
$A_{-\gamma}=A_{\gamma}^{\ast}$ and $A_{w\gamma}=\tau(w)
A_{\gamma}\tau\big( w^{-1}\big) $.
\end{Theorem}

\begin{proof}
The relation $\widetilde{A} ( \alpha,\beta,T,T^{\prime} )
=\widetilde{A} ( \beta,\alpha,T^{\prime},T ) $ shows $ (
A_{\alpha-\beta} ) _{T,T^{\prime}}= ( A_{\beta-\alpha} )
_{T^{\prime},T}$. By def\/i\-ni\-tion
\begin{gather*}
\big\langle w\big( x^{\alpha}\otimes T\big) ,w\big( x^{\beta}\otimes
T^{\prime}\big) \big\rangle _{\mathbb{T}}=\big\langle x^{w\alpha}%
\otimes\tau(w) T,x^{w\beta}\otimes\tau(w)
T^{\prime}\big\rangle _{\mathbb{T}}\\
\hphantom{\big\langle w\big( x^{\alpha}\otimes T\big) ,w\big( x^{\beta}\otimes
T^{\prime}\big) \big\rangle _{\mathbb{T}}}{}
=\big(\langle T,T\rangle _{0}\big\langle T^{\prime},T^{\prime
}\big\rangle _{0}\big) ^{1/2}T^{\ast}\tau(w) ^{\ast
}A_{w\alpha-w\beta}\tau(w) T^{\prime}\\
\hphantom{\big\langle w\big( x^{\alpha}\otimes T\big) ,w\big( x^{\beta}\otimes
T^{\prime}\big) \big\rangle _{\mathbb{T}}}{}
=\big\langle x^{\alpha}\otimes T,x^{\beta}\otimes T^{\prime}\big\rangle
_{\mathbb{T}}=\big(\langle T,T\rangle _{0}\big\langle
T^{\prime},T^{\prime}\big\rangle _{0}\big) ^{1/2}T^{\ast}A_{\alpha-\beta
}T^{\prime}
\end{gather*}
and thus $A_{\gamma}=\tau(w) ^{-1}A_{w\gamma}\tau (
w ) $ (recall $\tau$ is real-orthogonal so $\tau(w)
^{\ast}=\tau\big( w^{-1}\big) $).
\end{proof}

Summing over the graded components $\boldsymbol{Z}_{N,n}$ produces Laurent
polynomials with good properties, such as analyticity in $(
\mathbb{C}\backslash \{ 0 \} ) ^{N}$. The maps $a\mapsto
w\alpha$ ($w\in\mathcal{S}_{N}$) and $\alpha\mapsto-\alpha$ act as
permutations on each $\boldsymbol{Z}_{N,n}$.

\begin{Definition}
For $n=0,1,2,\ldots$ let
\begin{gather*}
H_{n}(x) :=\sum_{\alpha\in\boldsymbol{Z}_{N,n}}A_{\alpha
}x^{\alpha},
\end{gather*}
a Laurent polynomial with matrix coef\/f\/icients.
\end{Definition}

For complex Laurent polynomials $f(x) =\sum\limits_{\alpha
\in\mathbb{Z}^{N}}c_{\alpha}x^{\alpha}$ (f\/inite sum) def\/ine $f(
x) ^{\ast}=\sum\limits_{\alpha\in\mathbb{Z}^{N}}\overline{c_{\alpha}
}x^{-\alpha}$; if the coef\/f\/icients $\{ c_{\alpha}\} $ are
matrices then $f(x) ^{\ast}=\sum\limits_{\alpha\in\mathbb{Z}
^{N}}c_{\alpha}^{\ast}x^{-\alpha}$. There is a slight abuse of notation here:
if $x\in\mathbb{T}^{N}$ then $\overline{( x^{\alpha})
}=x^{-\alpha}$ and $f(x) ^{\ast}$ agrees with the adjoint of the
matrix $f(x)$.

\begin{Proposition}\label{Hprops} Suppose $n=0,1,2,\ldots$ and $w\in\mathcal{S}_{N}$ then
$H_{n}(xw) = \tau(w) ^{-1}H_{n} (x) \tau(w) $ and $H_{n}(x) ^{\ast}=H_{n}(x)$.
\end{Proposition}

\begin{proof}
Compute%
\begin{align*}
H_{n}(xw) & =\sum_{\alpha\in\boldsymbol{Z}_{N,n}}A_{\alpha
}(wx) ^{\alpha}=\sum_{\alpha\in\boldsymbol{Z}_{N,n}}A_{\alpha
}x^{w\alpha}=\sum_{\beta\in\boldsymbol{Z}_{N,n}}A_{w^{-1}\beta}x^{\beta}\\
& =\tau\big( w^{-1}\big) \sum_{\beta\in\boldsymbol{Z}_{N,n}}A_{\beta
}x^{\beta}\tau(w) =\tau(w) ^{-1}H_{n}(x) \tau(w) .
\end{align*}
Also $H_{n}(x) ^{\ast}=\sum\limits_{\alpha\in\boldsymbol{Z}%
_{N,n}}A_{\alpha}^{\ast}x^{-\alpha}=\sum\limits_{\alpha\in\boldsymbol{Z}%
_{N,n}}A_{-\alpha}x^{-\alpha}=H_{n}(x) $.
\end{proof}

As a consequence we f\/ind an important commutation satisf\/ied by a particular
point value of~$H_{n}(x) $ (recall the $N$-cycle $w_{0}=(
1,2,\ldots,N)$).

\begin{Corollary}
Suppose $n=1,2,3,\ldots$ then $\tau(w_{0}) ^{-1}H_{n} (x_{0}) \tau(w_{0}) =H_{n}(x_{0}) $, where
$x_{0}=\big( 1,\omega,\ldots,\omega^{N-1}\big) $, $\omega=\exp\frac
{2\pi\mathrm{i}}{N}$.
\end{Corollary}

\begin{proof}
By def\/inition $x_{0}w_{0}=\big( \omega,\ldots,\omega^{N-1},1\big) =\omega
x_{0}$. Each monomial $x^{\alpha}$ for $\alpha\in\boldsymbol{Z}_{N}$ is
homogeneous of degree zero (suppose $c\in\mathbb{C}\backslash\{
0\} $ then $(cx) ^{\alpha}=c^{\alpha_{1}+\cdots
+\alpha_{N}}x^{\alpha}=x^{\alpha}$) thus $H_{n}(x_{0}w_{0})
=H_{n}( \omega x_{0}) =H_{n}(x_{0}) $ and
$H_{n}(x_{0}w_{0}) =\tau(w_{0}) ^{-1}H_{n}(x_{0})\tau(w_{0}) $.
\end{proof}

We turn to the harmonic analysis signif\/icance of the matrices $\{
A_{\alpha}\} $. For an integrable func\-tion~$f$ on $\mathbb{T}^{N}$ the
Fourier transform (coef\/f\/icient) is
\begin{gather*}
\widehat{f}(\alpha) =\int_{\mathbb{T}^{N}}f(x)
x^{-\alpha}\mathrm{d}m(x) ,\qquad \alpha\in\mathbb{Z}^{N},
\end{gather*}
where $x:= ( \exp ( \mathrm{i}\theta_{1} ) ,\ldots,\exp (
\mathrm{i}\theta_{N} ) ) $ and $\mathrm{d}m(x)
=(2\pi) ^{-N}\mathrm{d}\theta_{1}\cdots\mathrm{d}\theta_{N}$;
and for a Baire measure $\mu$ on $\mathbb{T}^{N}$ the Fourier--Stieltjes
transform is%
\begin{gather*}
\widehat{\mu}(\alpha) =\int_{\mathbb{T}^{N}}x^{-\alpha
}\mathrm{d}\mu(x) , \qquad \alpha\in\mathbb{Z}^{N}.
\end{gather*}
We will show that there is a matrix-valued measure $\mu$, positive in a
certain sense, such that $\widehat{\mu}(\alpha) =A_{\alpha}$ for
all $\alpha\in\mathbb{Z}^{N}$ provided that $-1/h_{\tau}<\kappa<1/h_{\tau}$.
There is a~version of a theorem of Bochner about positive-def\/inite functions
on a~locally compact abelian group which proves this claim. The details of the
proof and some consequences are in Appendix~\ref{appendixA}.

Let $n=\dim V_{\tau}$ and identify $V_{\tau}$ with $\mathbb{C}^{n}$ whose
elements are considered as column vectors (in ef\/fect we use indices $1\leq
i\leq n$ instead of $\{ T\in\mathcal{Y}(\tau)\}
$). The inner product on $\mathbb{C}^{n}$ is $\langle u,v\rangle
:=\sum\limits_{i=1}^{n}\overline{u_{i}}v_{i}$, and the norm is $\vert
v\vert =\sqrt{\langle v,v\rangle }$. Note that $\langle
u,Av\rangle =u^{\ast}Av$. A positive-def\/inite matrix $P$ satisf\/ies
$\langle u,Pu\rangle \geq0$ for all $u\in\mathbb{C}^{n}$ (this
implies $P^{\ast}=P$).

\begin{Definition}
A function $F\colon \mathbb{Z}^{N}\rightarrow M_{n} ( \mathbb{C} ) $ is
positive-def\/inite if
\begin{gather*}
\sum\limits_{\alpha,\beta\in\mathbb{Z}^{N}}f(\alpha) ^{\ast
}F ( \alpha-\beta ) f(\beta) \geq0
\end{gather*}
for any f\/initely supported $\mathbb{C}^{n}$-valued function~$f$ on
$\mathbb{Z}^{N}$.
\end{Definition}

\begin{Theorem}
Suppose $F$ is positive-definite then there exist Baire measures $\{
\mu_{jk}\colon 1\leq j,k\leq n \} $ on $\mathbb{T}^{N}$ such that
\begin{gather*}
\int_{\mathbb{T}^{N}}x^{-\alpha}\mathrm{d}\mu_{jk}(x) =F (
\alpha ) _{jk},\qquad \alpha\in\mathbb{Z}^{N}, \qquad 1\leq j,k\leq n.
\end{gather*}
Furthermore each $\mu_{jj}$ is positive and
\begin{gather*}
\left\langle f,g\right\rangle _{F}:=\sum_{i,j=1}^{n}\int_{\mathbb{T}^{N}
}\overline{f(x) _{i}}g(x) _{j}\mathrm{d}\mu
_{ij}(x)
\end{gather*}
defines a positive-semidefinite Hermitian form on $C\big( \mathbb{T}^{N};\mathbb{C}^{n}\big) $ $($continuous $\mathbb{C}^{n}$-valued functions on
$\mathbb{T}^{N})$ satisfying $ \vert \langle f,g \rangle
_{F} \vert \leq B \Vert f \Vert _{\infty} \Vert g \Vert
_{\infty}$ for $f,g\in C \big( \mathbb{T}^{N};\mathbb{C}^{n} \big) $ with
$B<\infty$.
\end{Theorem}

The proof is in Appendix~\ref{BTheorem} and Theorem~\ref{Bformthm}. (In
general the measures $\mu_{jk}$ are not real-valued for $j\neq k$.) For
notational simplicity we introduce
\begin{gather}
\int_{\mathbb{T}^{N}}f(x) ^{\ast}\mathrm{d}\mu(x)
g(x) :=\sum_{i,j=1}^{n}\int_{\mathbb{T}^{N}}\overline{f(x) _{i}}g(x) _{j}\mathrm{d}\mu_{ij}(x) .
\label{intmat}
\end{gather}

To show that $\alpha\mapsto A_{\alpha}$ is positive-def\/inite let $f$ be a
f\/initely supported $\mathbb{C}^{n}$-valued function~$f$ on $\mathbb{Z}^{N}$
and let $p(x) =\sum\limits_{\alpha,T} \langle
T,T \rangle _{0}^{-1/2}f_{T}(\alpha) x^{\alpha}\otimes T$
be the associated Laurent polynomial (now we use the $T$ indices on
$\mathbb{C}^{n}$). Because this is a f\/inite sum there is a nonnegative integer
$m$ such that $e_{N}^{m}p(x) $ is polynomial (no negative
powers). Then for $-1/h_{\tau}<\kappa<1/h_{\tau}$
\begin{gather*}
0 \leq\big\langle e_{N}^{m}p,e_{N}^{m}p\big\rangle _{\mathbb{T}}
 = \!\sum_{\alpha,\beta\in\mathbb{Z}^{N}}\sum_{T,T^{\prime}}\big(
 \langle T,T \rangle _{0}\big\langle T^{\prime},T^{\prime
}\big\rangle _{0}\big) ^{-1/2}\overline{f_{T}(\alpha)
}f_{T^{\prime}}(\beta) \big\langle x^{\alpha+m\boldsymbol{1}
}\otimes T,x^{\beta+m\boldsymbol{1}}\otimes T^{\prime}\big\rangle \\
\hphantom{0 \leq\big\langle e_{N}^{m}p,e_{N}^{m}p\big\rangle _{\mathbb{T}}}
 =\sum_{\alpha,\beta\in\mathbb{Z}^{N}}\sum_{T,T^{\prime}}\overline
{f_{T}(\alpha) }f_{T^{\prime}}(\beta) (
A_{\alpha-\beta} ) _{T,T^{\prime}}.
\end{gather*}
Let $\mu= [ \mu_{T,T^{\prime}} ] $ be the matrix of measures
produced by the theorem, that is,
\begin{gather*}
( A_{\alpha}) _{T,T^{\prime}}=\int_{\mathbb{T}^{N}}x^{-\alpha
}\mathrm{d}\mu_{T,T^{\prime}}(x) ,\qquad \alpha\in\mathbb{Z}
^{N},\qquad T,T^{\prime}\in\mathcal{Y}(\tau) .
\end{gather*}

\begin{Theorem}
For $-1/h_{\tau}<\kappa<1/h_{\tau}$ there exists a matrix of Baire measures
$\mu= [ \mu_{T,T^{\prime}} ] $ on~$\mathbb{T}^{N}$ such that
\begin{gather*}
\langle f,g\rangle _{\mathbb{T}}=\int_{\mathbb{T}^{N}}f(
x) ^{\ast}\mathrm{d}\mu(x) g(x)
\end{gather*}
for all Laurent polynomials $f$, $g$ with coefficients in~$V_{\tau}$, in
particular for all NSJP's $f$,~$g$.
\end{Theorem}

Of course we want more detailed information about these measures. The f\/irst
step is to apply an approximate identity, a tool from the convolution
structure for measures and functions on the torus. We consider Ces\`{a}ro
summation of the series $\sum_{\alpha}A_{\alpha}x^{\alpha}$ based on summing
f\/irst over each $\boldsymbol{Z}_{N,n}$. Set%
\begin{gather*}
S_{n}(x) :=\sum_{\alpha\in\boldsymbol{Z}_{N,n}}x^{\alpha},
\end{gather*}
a Laurent polynomial, and the corresponding $( C,\delta) $-kernel
(for $\delta>0$) is def\/ined to be (the Pochhammer symbol is $(t)
_{m}=\prod\limits_{i=1}^{m} ( t+i-1 ) $)
\begin{gather*}
\sigma_{n}^{\delta}(x) :=\sum_{k=0}^{n}\frac{(-n)
_{k}}{(-n-\delta) _{k}}S_{k}(x) .
\end{gather*}
The point is that $\lim\limits_{n\rightarrow\infty}\frac{(-n)
_{k}}{(-n-\delta) _{k}}=1$ for f\/ixed $k$. In terms of
convolution
\begin{gather*}
\sigma_{n}^{\delta}\ast\mu(x) =\int_{\mathbb{T}^{N}}
\sigma_{n}^{\delta}\big( xy^{-1}\big) \mathrm{d}\mu(y), \\
\widehat{( \sigma_{n}^{\delta}\ast\mu) }(\alpha)
 =\int_{\mathbb{T}^{N}}\int_{\mathbb{T}^{N}}x^{-\alpha}\sigma_{n}^{\delta
}\big( xy^{-1}\big) \mathrm{d}\mu(y) \mathrm{d}m(x) \\
\hphantom{\widehat{( \sigma_{n}^{\delta}\ast\mu) }(\alpha)}{} =\int_{\mathbb{T}^{N}}\int_{\mathbb{T}^{N}}(xy) ^{-\alpha
}\sigma_{n}^{\delta}(x) \mathrm{d}\mu(y)
\mathrm{d}m(x) =A_{\alpha}\widehat{\sigma_{n}^{\delta}}(\alpha) ,
\end{gather*}
and $\widehat{\sigma_{n}^{\delta}}(\alpha) =\frac{(
-n) _{k}}{(-n-\delta) _{k}}$ for $\alpha\in
\boldsymbol{Z}_{N,k}$ for $0\leq k\leq n$ and $=0$ for $\vert
\alpha\vert >2n$ (or $\alpha\notin\boldsymbol{Z}_{N}$). Thus $\sigma
_{n}^{\delta}\ast\mu(x) =\sum\limits_{k=0}^{n}\frac{(
-n) _{k}}{(-n-\delta) _{k}}H_{k}(x) $. In
fact $\sigma_{n}^{N-1}(x) \geq0$ for all $x\in\mathbb{T}^{N}$
(Corollary~\ref{Cespos} below) which implies $\sigma_{n}^{N-1}\ast\mu$
converges to $\mu$ in a useful sense (weak-$\ast$, see Theorem \ref{propKn}(4)) and $\sigma_{n}^{N-1}\ast\mu(x) $ is a Laurent polynomial
all of whose point values are positive-semidef\/inite matrices. Also $\big\Vert
\sigma_{n}^{N-1}\big\Vert _{1}:=\int_{\mathbb{T}^{N}}\big\vert \sigma
_{n}^{N-1}\big\vert \mathrm{d}m=1$.

The \textit{complete symmetric polynomial} in $N$ variables and degree $n$ is
given by
\begin{gather*}
h_{n}(x) :=\sum\left\{ x^{\alpha}\colon \alpha\in\mathbb{N}_{0}
^{N}\colon \sum_{i=1}^{N}\alpha_{i}=n\right\} .
\end{gather*}
Recall
\begin{gather*}
\#\left\{ \alpha\in\mathbb{N}_{0}^{N}\colon \sum_{i=1}^{N}\alpha
_{i}=m\right\} =\frac{(N) _{m}}{m!} \qquad \text{for} \quad m=0,1,2,3,\ldots.
\end{gather*}

\begin{Theorem}
For $n\geq0$
\begin{gather}
h_{n}\left( \frac{1}{x_{1}},\frac{1}{x_{2}},\ldots,\frac{1}{x_{N}}\right)
h_{n} ( x_{1},\ldots,x_{N} ) =\frac{(N) _{n}}
{n!}\sigma_{n}^{N-1}(x) .
\end{gather}
\end{Theorem}

\begin{proof}
The product is a sum of terms $x^{\alpha-\beta}$ with $\alpha,\beta
\in\mathbb{N}_{0}^{N}$ and $\vert\alpha\vert =n= \vert
\beta \vert $. For example the term $x^{0}=1$ appears exactly
$\frac{(N) _{n}}{n!}$ times, because the number of terms in
$h_{n}$ is $\frac{(N) _{n}}{n!}$. Consider a~f\/ixed $\gamma
\in\boldsymbol{Z}_{N,m}$ for some $m$ with $0\leq m\leq n$. The term
$x^{\gamma}$ appears in the product for each pair $( \alpha
,\beta) $ with
\begin{gather*}
\alpha=\gamma^{\pi}+\alpha^{\prime},\qquad \beta=\gamma^{\nu}+\alpha^{\prime}, \qquad \gamma=\alpha-\beta,
\end{gather*}
where $\alpha^{\prime}\in\mathbb{N}_{0}^{N}$ and $\sum\limits_{i=1}^{N}\alpha
_{i}^{^{\prime}}=n-m$. Recall $\gamma_{i}^{\pi}=\max( \gamma
_{i},0) $ and $\gamma_{i}^{\nu}=-\min ( 0,\gamma_{i} )
=\max ( 0,-\gamma_{i} ) $; thus $\sum\limits_{i=1}^{N}\gamma_{i}^{\pi}=m$
and $\sum\limits_{i=1}^{N}\alpha_{i}=n$. Therefore the coef\/f\/icient of $x^{\gamma}$ is
\begin{gather*}
\#\left\{ \alpha^{\prime}\in\mathbb{N}_{0}^{N}\colon \sum\limits_{i=1}^{N}\alpha
_{i}^{^{\prime}}=n-m\right\} =\frac{(N) _{n-m}}{(n-m) !}.
\end{gather*}
Hence
\begin{gather*}
h_{n}\left( \frac{1}{x_{1}},\frac{1}{x_{2}},\ldots,\frac{1}{x_{N}}\right)
h_{n}( x_{1},\ldots,x_{N}) =\sum_{m=0}^{N}\frac{(N) _{n-m}}{(n-m) !}S_{m}(x) .
\end{gather*}
To f\/inish the proof multiply this relation by $\frac{n!}{(N)
_{n}}$ and compute
\begin{gather*}
\frac{n!}{(N) _{n}}\frac{(N) _{n-m}}{(n-m) !} =(-1) ^{m}(-n) _{m}\frac{(N) _{n-m}}{(N) _{n-m}(N+n-m)
_{m}}\\
\hphantom{\frac{n!}{(N) _{n}}\frac{(N) _{n-m}}{(n-m) !}}{}
=(-1) ^{m}\frac{(-n) _{m}}{(N+n-m) _{m}}=\frac{(-n) _{m}}{( 1-N-n) _{m}}.\tag*{\qed}
\end{gather*}
\renewcommand{\qed}{}
\end{proof}

\begin{Corollary}
\label{Cespos}$\sigma_{n}^{N-1}(x) \geq0$ for all $x\in
\mathbb{T}^{N}$.
\end{Corollary}

\begin{proof}
$\overline{h_{n}( x_{1},\ldots,x_{N}) }=h_{n}\left( \frac
{1}{x_{1}},\frac{1}{x_{2}},\ldots,\frac{1}{x_{N}}\right) $ for $x\in
\mathbb{T}^{N}$.
\end{proof}

Observe that this kernel applies to the quotient space $\mathbb{T}%
^{N}/\mathbb{D}$ where%
\begin{gather*}
\mathbb{D}:= \{ ( u,u,\ldots,u ) \colon u\in\mathbb{C},\, \vert
u \vert =1 \} ,
\end{gather*}
is the diagonal subgroup. That is, each $S_{n}(x) $ is
homogeneous of degree zero, constant on sets $ \{ ( ux_{1},ux_{2},\ldots,ux_{N} ) \colon \vert u \vert =1 \} $ for f\/ixed
$x\in\mathbb{T}^{N}$.

Here are approximate identity properties of $\sigma_{n}^{N-1}$; we use
$\mathbb{T}^{N}/\mathbb{D}$ to refer to functions homogeneous of degree zero.
There is a standard formula:

\begin{Lemma}
\label{convghnu}Suppose $g,h\in C\big( \mathbb{T}^{N}\big) $ and $\nu$ is
a Baire measure on $\mathbb{T}^{N}$ then for $h^{\dag}(x)
:=h\big( x^{-1}\big) $
\begin{gather*}
\int_{\mathbb{T}^{N}}g(x) ( h\ast\nu ) (
x ) \mathrm{d}m(x) =\int_{\mathbb{T}^{N}}\big( g\ast
h^{\dag}\big) (y) \mathrm{d}\nu(y) .
\end{gather*}
\end{Lemma}

\begin{proof}
The left side equals
\begin{gather*}
\int_{\mathbb{T}^{N}}\int_{\mathbb{T}^{N}}g(x) h\big(
xy^{-1}\big) \mathrm{d}\nu(y) \mathrm{d}m(x)
=\int_{\mathbb{T}^{N}}\int_{\mathbb{T}^{N}}g(x) h^{\dag}\big(
yx^{-1}\big) \mathrm{d}m(x) \mathrm{d}\nu(y)
\end{gather*}
(by Fubini's theorem) which equals the right side.
\end{proof}

The following is a standard result on approximate identities.

\begin{Proposition}\label{appxC} Suppose $f\in C\big( \mathbb{T}^{N}/\mathbb{D}\big) $ then
$\big\Vert f-f\ast\sigma_{n}^{N-1}\big\Vert _{\infty}\rightarrow0$ as
$n\rightarrow\infty$.
\end{Proposition}

\begin{proof}
For $\varepsilon>0$ there exists a Laurent polynomial $p$ on $\mathbb{T}
^{N}/\mathbb{D}$ such that $\Vert f-p\Vert _{\infty}<\varepsilon$.
Then%
\begin{gather*}
f-f\ast\sigma_{n}^{N-1}=(f-p) +\big( p-\sigma_{n}^{N-1}\ast
p\big) +(p-f) \ast\sigma_{n}^{N-1},
\end{gather*}
and $\big\Vert (p-f) \ast\sigma_{n}^{N-1}\big\Vert _{\infty
}\leq \Vert p-f \Vert _{\infty}\big\Vert \sigma_{n}^{N-1}\big\Vert
_{1}<\varepsilon$. Let $p(x) =\sum\limits_{m=0}^{M}%
\sum\limits_{\alpha\in\boldsymbol{Z}_{N,m}}c_{\alpha}x^{\alpha}$ for some
coef\/f\/icients~$c_{\alpha}$ (and f\/inite~$M$); thus
\begin{gather*}
\big( p-\sigma_{n}^{N-1}\ast p\big) (x) =\sum
\limits_{m=0}^{M}\sum\limits_{\alpha\in\boldsymbol{Z}_{N,m}}\left(
1-\frac{(-n) _{m}}{(1-N-n) _{m}}\right)
c_{\alpha}x^{\alpha},
\end{gather*}
which tends to zero in norm as $n\rightarrow\infty$.
\end{proof}

\begin{Corollary}
Suppose $\nu$ is a Baire measure on $\mathbb{T}^{N}$ and $f\in C\big(
\mathbb{T}^{N}/\mathbb{D}\big) $ then
\begin{gather*}
\lim_{n\rightarrow\infty}\int_{\mathbb{T}^{N}}f(x) \big(
\sigma_{n}^{N-1}\ast\nu\big) (x) \mathrm{d}m(x)
=\int_{\mathbb{T}^{N}}f(x) \mathrm{d}\nu(x) .
\end{gather*}
\end{Corollary}

\begin{proof}
By Lemma \ref{convghnu}
\begin{gather*}
\int_{\mathbb{T}^{N}}f(x) \big( \sigma_{n}^{N-1}\ast
\nu\big) (x) \mathrm{d}m(x) =\int
_{\mathbb{T}^{N}}\big( f\ast\sigma_{n}^{N-1}\big) (x)
\mathrm{d}\nu(x) ,
\end{gather*}
since $\big( \sigma_{n}^{N-1}\big) ^{\dag}=\sigma_{n}^{N-1}$, and
$f\ast\sigma_{n}^{N-1}$ converges uniformly to $f$ as $n\rightarrow\infty$.
\end{proof}

\begin{Definition}
Def\/ine the $\mu$-approximating Laurent polynomials
\begin{gather*}
K_{n}(x) :=\sigma_{n}^{N-1}\ast\mu(x) =\sum
_{m=0}^{n}\frac{(-n) _{m}}{(1-N-n) _{m}}\sum_{\alpha\in\boldsymbol{Z}_{N,n}}A_{\alpha}x^{\alpha}.
\end{gather*}
\end{Definition}

Note $\frac{(-n) _{m}}{(1-N-n) _{m}}
=\frac{(n-m+1) _{N-1}}{(n+1) _{N-1}}$ for $0\leq
m\leq n$; for example with $N=3$, $\frac{(-n) _{m}}{(
-2-n) _{m}}=\big( 1-\frac{m}{n+1}\big) \big( 1-\frac{m}{n+2}\big)$, and $=0$ for $m>n$.

\begin{Theorem}
\label{propKn}
For $-1/h_{\tau}<\kappa<1/h_{\tau}$ and $n=1,2,3,\ldots$ the
following hold:
\begin{enumerate}\itemsep=0pt
\item[$(1)$] $K_{n}(x) $ is positive semi-definite
for each $x\in\mathbb{T}^{N}$,
\item[$(2)$] $K_{n}(xw)
=\tau(w) ^{-1}K_{n}(x) \tau(w) $ for
each $x\in\mathbb{T}^{N},w\in\mathcal{S}_{N}$,
\item[$(3)$] $K_{n}(x_{0}) \tau(w_{0}) =\tau(w_{0})
K_{n}(x_{0})$,
\item[$(4)$] $\lim\limits_{n\rightarrow\infty}
\int_{\mathbb{T}^{N}}f(x) ^{\ast}K_{n}(x) g(
x) \mathrm{d}m(x) = \langle f,g \rangle
_{\mathbb{T}}$ for all $f,g\in\mathcal{P}_{\tau}$; the limit exists for any
$f,g\in C\big( \mathbb{T}^{N};V_{\tau}\big) $ and defines a $ \Vert
\cdot \Vert _{\infty}$-bounded positive Hermitian form.
\end{enumerate}
\end{Theorem}

\begin{proof}
Part (1) is a consequence of Theorem~\ref{posdefpt}. Parts~(2)
and~(3) follow from the properties of $H_{m}$ in Proposition \ref{Hprops}. For
part~(4) there is an intermediate step of averaging over the diagonal group~$\mathbb{D}$. Def\/ine the operator
$\rho\colon C\big( \mathbb{T}^{N}\big)
\rightarrow C\big( \mathbb{T}^{N}/\mathbb{D}\big) $ by
\begin{gather*}
\rho(p) (x) :=\frac{1}{2\pi}\int_{-\pi}^{\pi
}p\big( e^{\mathrm{i}\theta}x\big) \mathrm{d}\theta, \qquad p\in C\big(
\mathbb{T}^{N}\big).
\end{gather*}
Clearly $ \Vert \rho(p) \Vert _{\infty}\leq \Vert
p \Vert _{\infty}$; in ef\/fect $\rho$ is the projection onto Fourier
series supported by~$\boldsymbol{Z}_{N}$. Then
\begin{gather*}
\begin{split}
& \int_{\mathbb{T}^{N}}p(x) \mathrm{d}\mu_{T,T^{\prime}} (
x ) =\int_{\mathbb{T}^{N}}\rho(p) (x)
\mathrm{d}\mu_{T,T^{\prime}}(x) ,\\
& \int_{\mathbb{T}^{N}}p(x) ( K_{n}(x)
) _{T,T^{\prime}}\mathrm{d}m(x) =\int_{\mathbb{T}^{N}}\rho(p) (x) ( K_{n}(x)
 ) _{T,T^{\prime}}\mathrm{d}m(x) , \qquad T,T^{\prime}\in\mathcal{Y}(\tau) .
 \end{split}
\end{gather*}
To extend this to the form $ \langle \cdot,\cdot \rangle
_{\mathbb{T}}$ express the typical sum $\sum\limits_{i,j=1}^{n}\overline{f_{i}}
B_{ij}g_{j}=\operatorname{tr}\big( ( f\otimes g^{\ast} ) ^{\ast}B\big) $
where $\operatorname{tr}$ denotes the trace and $( f\otimes g^{\ast})
_{ij}=f_{i}\overline{g_{j}}$ ($1\leq i,j\leq n$). Then ($\rho$ is applied to
matrices entry-wise) for $f,g\in C\big( \mathbb{T}^{N};V_{\tau}\big) $
\begin{gather*}
\int_{\mathbb{T}^{N}}f(x) ^{\ast}\mathrm{d}\mu(x)
g(x) =\int_{\mathbb{T}^{N}}\operatorname{tr}\big[ ( f(x) \otimes g(x) ^{\ast}) ^{\ast}\mathrm{d}\mu(x) \big]
 =\int_{\mathbb{T}^{N}}\operatorname{tr}\big[ \{ \rho ( f(x)
\otimes g(x) ^{\ast} ) \} ^{\ast}\mathrm{d}\mu(x) \big] ,\\
\int_{\mathbb{T}^{N}}f(x) ^{\ast}K_{n}(x) g(
x) \mathrm{d}m(x) =\int_{\mathbb{T}^{N}}\operatorname{tr}\big[
 ( f(x) \otimes g(x) ^{\ast} ) ^{\ast}K_{n}(x) \big] \mathrm{d}m(x) \\
 \hphantom{\int_{\mathbb{T}^{N}}f(x) ^{\ast}K_{n}(x) g(x) \mathrm{d}m(x)}{}
 =\int_{\mathbb{T}^{N}}\operatorname{tr} \big[ \{ \rho ( f(x)
\otimes g(x) ^{\ast} ) \} ^{\ast}K_{n} (
x ) \big] \mathrm{d}m(x) .
\end{gather*}
The convergence properties of Proposition~\ref{appxC} imply part~(4).
\end{proof}

\section{Recurrence relations}\label{section5}

As a simple illustration consider $\boldsymbol{Z}_{N,1}$ where it suf\/f\/ices to
f\/ind $A_{1,-1,0\ldots,0}$. Introduce the unit basis vectors $\varepsilon_{i}$
for $\mathbb{Z}^{N}$ (with $( \varepsilon_{i}) _{j}=\delta_{ij}$), so that $( 1,-1,0,\ldots) =\varepsilon_{1}-\varepsilon_{2}$. The
relation $( \varepsilon_{2}-\varepsilon_{1}) =( 1,2)
( \varepsilon_{1}-\varepsilon_{2}) $ implies $A_{\varepsilon
_{2}-\varepsilon_{1}}=\tau(1,2) A_{\varepsilon_{1}
-\varepsilon_{2}}\tau(1,2) =A_{\varepsilon_{1}-\varepsilon_{2}
}^{\ast}$. From $\langle x_{1}\mathcal{D}_{1}f,g\rangle
_{\mathbb{T}}=\langle f,x_{1}\mathcal{D}_{1}g\rangle _{\mathbb{T}}$
(Proposition \ref{torusform}(3)) we f\/ind
\begin{gather*}
x_{1}\mathcal{D}_{1} ( x_{1}\otimes T ) =x_{1}\otimes T+\kappa
x_{1}\sum_{j=2}^{N}\frac{x_{1}-( x(1,j)) _{1}}{x_{1}-x_{j}}\tau(1,j) T\\
\hphantom{x_{1}\mathcal{D}_{1} ( x_{1}\otimes T )}{}
 =x_{1}\otimes T+\kappa x_{1}\otimes\tau(\omega_{1})
T=x_{1}\otimes ( I+\kappa\tau(\omega_{1}) ) T,\\
x_{1}\mathcal{D}_{1} ( x_{2}\otimes T^{\prime} ) =\kappa
x_{1}\sum_{j=2}^{N}\frac{x_{2}-( x(2,j) ) _{1}}{x_{1}-x_{j}}\tau(1,j) T^{\prime}=-\kappa x_{1}\tau(1,2) T^{\prime};
\end{gather*}
recall the Jucys--Murphy elements $\omega_{i}:=\sum\limits_{j=i+1}^{N}(i,j) $ and the action $\tau(\omega_{i}) T=c(
i,T) T$ for \mbox{$T\in\mathcal{Y}(\tau) $}. Next the equation
$\langle x_{1}\mathcal{D}_{1}( x_{1}\otimes T) ,x_{2}\otimes
T^{\prime}\rangle _{\mathbb{T}}=\langle x_{1}\otimes T,x_{1}\mathcal{D}_{1}( x_{2}\otimes T^{\prime}) \rangle
_{\mathbb{T}}$ yields
\begin{gather*}
T^{\ast} ( I+\kappa\tau(\omega_{1}) ) ^{\ast
}A_{\varepsilon_{1}-\varepsilon_{2}}T^{\prime}=-\kappa T^{\ast}A_{\mathbf{0}}\tau(1,2) T^{\prime},
\end{gather*}
and $A_{\mathbf{0}}=I$. This holds for arbitrary $T$, $T^{\prime}$, and
$\tau(\omega_{1}) $ is diagonal with the entry at $(T,T) $ being $c(1,T) $ thus
\begin{gather*}
( I+\kappa\tau(\omega_{1}) ) A_{\varepsilon
_{1}-\varepsilon_{2}} =-\kappa\tau(1,2), \qquad
A_{\varepsilon_{1}-\varepsilon_{2}} =-\kappa ( I+\kappa\tau (\omega_{1})) ^{-1}\tau(1,2) ,
\end{gather*}
provided $\kappa c(1,T) \neq-1$ for all $T\in\mathcal{Y}(\tau)$.

\begin{Lemma}
For $\alpha\in\mathbb{N}_{0}^{N}$, $T\in\mathcal{Y}(\tau) $ and $1\leq i\leq N$
\begin{gather*}
\begin{split}
& x_{i}\mathcal{D}_{i}\big( x^{\alpha}\otimes T\big) =\alpha
_{i}x^{\alpha}\otimes T-\kappa\sum_{\alpha_{j}>\alpha_{i}}\sum_{\ell
=1}^{\alpha_{j}-\alpha_{i}}x^{\alpha+\ell ( \varepsilon_{i}-\varepsilon_{j} ) }\otimes\tau(i,j) T\\
& \hphantom{x_{i}\mathcal{D}_{i}\big( x^{\alpha}\otimes T\big) =}{}
+\kappa\sum_{\alpha_{i}>\alpha_{j}}\sum_{\ell=0}^{\alpha_{i}-\alpha_{j}%
-1}x^{\alpha+\ell( \varepsilon_{j}-\varepsilon_{i}) }\otimes\tau(i,j) T.
\end{split}
\end{gather*}
\end{Lemma}

\begin{proof}
This follows from performing the division in $\frac{x^{\alpha}-x^{(i,j) \alpha}}{x_{i}-x_{j}}$.
\end{proof}

\begin{Proposition}
For $\alpha,\beta\in\mathbb{N}_{0}^{N}$ such that $\vert \alpha
\vert = \vert \beta \vert $ and $1\leq i\leq N$
\begin{gather}
( \alpha_{i}-\beta_{i}) A_{\alpha-\beta}
 =\kappa\sum_{\alpha_{j}>\alpha_{i}}\sum_{\ell=1}^{\alpha_{j}-\alpha_{i}
}\tau(i,j) A_{\alpha+\ell ( \varepsilon_{i}-\varepsilon
_{j} ) -\beta}-\kappa\sum_{\alpha_{i}>\alpha_{j}}\sum_{\ell=0}^{\alpha_{i}-\alpha_{j}-1}\tau(i,j) A_{\alpha+\ell (
\varepsilon_{j}-\varepsilon_{i} ) -\beta}\nonumber\\
\hphantom{( \alpha_{i}-\beta_{i}) A_{\alpha-\beta}=}{}
 -\kappa\sum_{\beta_{j}>\beta_{i}}\sum_{\ell=1}^{\beta_{j}-\beta_{i}
}A_{\alpha-\ell ( \varepsilon_{i}-\varepsilon_{j} ) -\beta}
\tau(i,j) \nonumber\\
\hphantom{( \alpha_{i}-\beta_{i}) A_{\alpha-\beta}=}{}
+\kappa\sum_{\beta_{i}>\beta_{j}}\sum_{\ell=0}
^{\beta_{i}-\beta_{j}-1}A_{\alpha-\ell( \varepsilon_{j}-\varepsilon
_{i}) -\beta}\tau(i,j) . \label{A(a-b)}
\end{gather}
\end{Proposition}

\begin{proof}
The statement follows from the equation
\begin{gather*}
\big\langle x_{i}\mathcal{D}_{i}( x^{\alpha}\otimes T) ,x^{\beta}\otimes T^{\prime
}\big\rangle =\big\langle x^{\alpha}\otimes T,x_{i}\mathcal{D}_{i}\big(
x^{\beta}\otimes T^{\prime}\big) \big\rangle
\end{gather*}
and the lemma.
\end{proof}

The following is one of the main results of this section. Note that it is
important to involve the multiplicity of the f\/irst part of $\gamma$.

\begin{Theorem}
Suppose $\gamma\in\boldsymbol{Z}_{N,n}$ such that $\gamma^{\pi}$ is a
partition and $\gamma_{1}^{\pi}=\gamma_{m}^{\pi}>\gamma_{m+1}$ then
\begin{gather}
\left( \gamma_{1}I+\kappa\sum_{\ell=m+1}^{N}\tau(1,\ell)
\right) A_{\gamma}=-\kappa\sum_{j=m+1,\, \gamma_{j}\geq0}^{N}\sum_{\ell
=1}^{\gamma_{1}-\gamma_{j}-1}\tau(1,j) A_{\gamma+\ell(
\varepsilon_{j}-\varepsilon_{1}) }\nonumber\\
\qquad{}
-\kappa\sum_{j=m+1,\, \gamma_{j}<0}^{N}\left\{ \tau(1,j)
\sum_{\ell=1}^{\gamma_{1}-1}A_{\gamma+\ell ( \varepsilon_{j}
-\varepsilon_{1}) }+\sum_{\ell=1}^{-\gamma_{j}}A_{\gamma-\ell(
\varepsilon_{1}-\varepsilon_{j}) }\tau(1,j) \right\}.\label{eqnApart}
\end{gather}
Each of the coefficients $A_{\delta}$ appearing on the second line of the
equation satisfies $\delta\in\boldsymbol{Z}_{N,s}$ for some $s<n$ and for~$A_{\delta}$ on the right-hand side of the first line $\delta=\delta^{\pi}-\gamma^{\nu}$ where $\delta^{\pi}\vartriangleleft\gamma^{\pi}$.
\end{Theorem}

\begin{proof}
The formula follows from equation~(\ref{A(a-b)}) by setting $i=1$, $\beta
_{1}=0$ and omitting the case $\alpha_{j}>\alpha_{i}$. Suppose $\gamma
_{j}^{\nu}=0$ for $j\leq k$, and $\gamma_{j}^{\pi}=0$ for $j>k$. The typical
multi-index in the second line is
\begin{gather*}
\big( \gamma_{1}-\ell,\gamma_{2}^{\pi},\ldots,\gamma_{k}^{\pi},-\gamma
_{k+1}^{\nu},\ldots,-\gamma_{j}^{\nu}+\ell,\ldots\big)
\end{gather*}
with $1\leq\ell\leq\gamma_{1}-1$ or $1\leq\ell\leq\gamma_{j}^{\nu}$. If
$\gamma_{1}\geq\ell$ then the sum of the nonnegative components is $n-\ell<n$;
and if $\gamma_{1}<\ell$ (possible if $\gamma_{j}^{\nu}>\gamma_{1}$) then the
sum of the nonnegative components is $n-\gamma_{1}$. In both cases the
multi-index is in $\bigcup\limits_{s<n}\boldsymbol{Z}_{N,s}$. From~\cite[Lemma~10.1.3]{Dunkl/Xu2014} $\gamma^{\pi}\succ ( \gamma^{\pi}+\ell ( \varepsilon
_{j}-\varepsilon_{1} ) ) ^{+}$ for $1\leq\ell\leq\gamma
_{1}-\gamma_{j}^{\pi}-1$, thus the multi-indices on the right-hand side of the
f\/irst line satisfy $\delta=\delta^{\pi}-\gamma^{\nu}$ with $(\delta^{\pi}) ^{+}\prec\gamma^{\pi}$, that is, $\delta^{\pi
}\vartriangleleft\gamma^{\pi}$.
\end{proof}

If $\gamma$ is $\vartriangleright$-minimal with f\/ixed $\gamma^{\nu}$ then
there are no terms on the right side of the f\/irst line (that is, $\gamma
_{j}\geq0$ implies $\gamma_{1}\geq\gamma_{j}\geq\gamma_{1}-1$).

\begin{Proposition}\label{mingam0} Among $\gamma\in\boldsymbol{Z}_{N,n}$ such that $\gamma^{\nu
}=\beta$ for some fixed $\beta$ with $ \vert \beta \vert =n$ and such
that $\beta_{j}>0$ exactly when $j>k$ the minimal multi-index for the order
$\gamma^{(1) \pi}\vartriangleright\gamma^{(2) \pi
}$ is $\gamma^{(0) }= \big( p+1,\ldots,\overset{(m) }{p+1},p,\ldots,\overset{(k) }{p},-\beta_{k+1}
,\ldots,-\beta_{N} \big) $ where $p=\big\lfloor \frac{n}{k}\big\rfloor $
and $m=n-kp$ $($so $0\leq m<k)$. For this multi-index the right-hand side of~\eqref{eqnApart} contains only $A_{\delta}$ with $\delta\in\bigcup
\limits_{s=0}^{n-1}\boldsymbol{Z}_{N,s}$.
\end{Proposition}

The proof is technical and is presented as Proposition~\ref{mingam}.

\begin{Theorem}
The coefficients $A_{\alpha}$ are rational functions of $\kappa$ and are
finite provided
\begin{gather*}
\kappa\notin\left\{ -\frac{m}{c}\colon m,c\in\mathbb{N},\, 1\leq
c\leq\tau_{1}-1\right\} \cup\left\{ \frac{m}{c}\colon m,c\in\mathbb{N},\, 1\leq
c\leq\ell(\tau) -1\right\} .
\end{gather*}
 Also $A_{-\alpha}=A_{\alpha
}^{\ast}$ and $\tau(w) ^{\ast}A_{w\alpha}\tau(w)
=A_{\alpha}$ for all $\alpha\in\boldsymbol{Z}_{N}$, $w\in\mathcal{S}_{N}$ and
permitted values of $\kappa$.
\end{Theorem}

\begin{proof}
The NSJP $\zeta_{\alpha,T}$ is a rational function of $\kappa$ with no poles
in $-1/h_{\tau}<\kappa<1/h_{\tau}$. The coef\/f\/icients $A_{\alpha}$ are def\/ined
in terms of all the NSJP's and are also rational in~$\kappa$. In equation~(\ref{eqnApart}) the operator on the left of~$A_{\gamma}$ is
\begin{gather*}
\left( \gamma_{1}I+\kappa\sum_{\ell=m+1}^{N}\tau(1,\ell)
\right) =\tau(1,m) ( \gamma_{1}I+\kappa\tau ( \omega_{m}) ) \tau(1,m) ,
\end{gather*}
where $\omega_{m}$ is the Jucys--Murphy element $\sum\limits_{\ell=m+1}^{N} (
m,\ell ) $; the action $\tau ( \omega_{m} ) T=c (
m,T ) T$ for all $T\in\mathcal{Y}(\tau) $ shows that the
eigenvalues of the operator are $ \{ \gamma_{1}+\kappa c (
m,T ) \colon T\in\mathcal{Y}(\tau) \} $ and the operator
is invertible provided $\kappa c ( m,T ) \notin \{
-1,-2,-3,\ldots \} $ for $1\leq m\leq N$. The set of values of~$c (m,T ) $ is $\{ j\in\mathbb{Z}\colon 1-\ell(\tau) \leq
j\leq\tau_{1}-1 \} $. Thus an inductive argument based on~$n$ in~$\boldsymbol{Z}_{N,n}$, the order in Proposition~\ref{mingam0}, and formula~(\ref{eqnApart}) shows there are unique solutions for~$\{ A_{\alpha
}\} $ provided that the possible poles at $n+\kappa c(
i,T) =0$ are excluded. The relations $A_{-\alpha}=A_{\alpha}^{\ast}$
and $\tau(w) ^{\ast}A_{w\alpha}\tau(w) =A_{\alpha
}$ hold at least in an interval hence for all $\kappa$, excluding the poles.
\end{proof}

The largest interval around $0$ without poles is $-\frac{1}{\tau_{1}-1}<\kappa<\frac{1}{\ell(\tau) -1}$. As illustration we describe~$A_{\gamma}$ for $\gamma\in\boldsymbol{Z}_{N,2}$. Above we showed
\begin{gather*}
A_{\varepsilon_{1}-\varepsilon_{j}}=-\kappa\ ( I+\kappa\tau(
\omega_{1} ) ) ^{-1}\tau(1,j) ,\qquad 2\leq j\leq N.
\end{gather*}
Next for $\alpha=\varepsilon_{1}+\varepsilon_{2}$ and $\beta=2\varepsilon_{j}$
for $3\leq j\leq N$ we f\/ind
\begin{gather*}
\left( I+\kappa\sum_{i=3}^{N}\tau(1,i) \right) A_{\varepsilon
_{1}+\varepsilon_{2}-2\varepsilon_{j}}=-\kappa ( A_{\varepsilon
_{2}-\varepsilon_{1}}+A_{\varepsilon_{2}-\varepsilon_{j}} ) \tau (1,j) .
\end{gather*}
For $\alpha=\varepsilon_{1}+\varepsilon_{2}$ and $\beta=\varepsilon
_{j}+\varepsilon_{j+1}$ with $3\leq j\leq N-1$
\begin{gather*}
\left( I+\kappa\sum_{i=3}^{N}\tau(1,i) \right) A_{\varepsilon
_{1}+\varepsilon_{2}-\varepsilon_{j}-\varepsilon_{j+1}}=-\kappa\big(
A_{\varepsilon_{2}-\varepsilon_{j}}\tau ( 1,j+1 ) +A_{\varepsilon
_{2}-\varepsilon_{j+1}}\tau(1,j) \big) .
\end{gather*}
For $\alpha=2\varepsilon_{1}$ and $\beta=2\varepsilon_{N}$ we obtain
\begin{gather*}
( 2I+\kappa\tau(\omega_{1})) A_{2\varepsilon
_{1}-2\varepsilon_{N}} \\
\qquad{}
 =-\kappa\left\{ \sum_{\ell=2}^{N-1}\tau(1,\ell)
A_{\varepsilon_{1}+\varepsilon_{\ell}-2\varepsilon_{N}}+\tau (
1,N) A_{\varepsilon_{1}-\varepsilon_{N}}+ ( A_{\varepsilon
_{1}-\varepsilon_{N}}+I ) \tau(1,N) \right\} .
\end{gather*}
The other coef\/f\/icients for $n=2$ are obtained using the relations
\begin{gather*}
A_{-\alpha
}=A_{\alpha}^{\ast} \qquad \text{and} \qquad \tau(w) ^{\ast}A_{w\alpha}\tau(
w) =A_{\alpha}.
\end{gather*}

\section{The dif\/ferential equation}\label{section6}

We will show that $\mu$ satisf\/ies a dif\/ferential system in a distributional
sense. Let $\mathbb{T}_{\rm reg}^{N}:=\mathbb{T}^{N}\backslash\bigcup
\limits_{1\leq i<j\leq N} \{ x\colon x_{i}=x_{j} \} $ (this avoids the
singularities of the system) and $\partial_{j}:=\frac{\partial}{\partial
x_{j}}$ for $1\leq j\leq N$. The system is
\begin{gather}
x_{i}\partial_{i}K(x) =\kappa\sum_{j\neq i}\frac{x_{j}}
{x_{i}-x_{j}}\tau(i,j) K(x) +\kappa K(x) \sum_{j\neq i}\tau(i,j) \frac{x_{i}}{x_{i}-x_{j}}, \qquad 1\leq i\leq N. \label{dsystK}
\end{gather}
Any solution of this system satisf\/ies $\sum\limits_{i=1}^{N}x_{i}\partial_{i}K\left(
x\right) =0$ and thus is homogeneous of degree zero.

The relation $\langle x_{i}\mathcal{D}_{i}f,g \rangle _{\mathbb{T}
}= \langle f,x_{i}\mathcal{D}_{i}g \rangle _{\mathbb{T}}$ extends to
$C^{(1) }\big( \mathbb{T}^{N};V_{\tau}\big) $, the
continuously dif\/fe\-ren\-tiab\-le $V_{\tau}$-valued functions, because Laurent
polynomials are dense in this space. We have shown
\begin{gather*}
\lim_{n\rightarrow\infty}\int_{\mathbb{T}^{N}}\big\{ ( x_{i}\mathcal{D}_{i}f(x) ) ^{\ast}K_{n}(x)
g(x) -f(x) ^{\ast}K_{n}(x)x_{i}\mathcal{D}_{i}g(x) \big\} \mathrm{d}m(x)=0.
\end{gather*}
Suppose $p,q\in C^{(1) }\big( \mathbb{T}^{N}\big) $ (scalar
$\mathbb{C}$-valued) then by periodicity $\int_{\mathbb{T}^{N}}\frac{\partial
}{\partial\theta_{j}}(pq) \mathrm{d}m=0$ thus
\begin{gather*}
\int_{\mathbb{T}^{N}}\left( \frac{\partial}{\partial\theta_{j}}p\right)
q\mathrm{d}m=-\int_{\mathbb{T}^{N}}p\left( \frac{\partial}{\partial\theta
_{j}}q\right) \mathrm{d}m.
\end{gather*}
Also from $\frac{\partial}{\partial\theta_{j}}f(x)
=\mathrm{i}e^{\mathrm{i}\theta_{j}}\partial_{j}f(x)
=\mathrm{i}x_{j}\partial_{j}f(x) $ and $\partial_{j}f^{\ast
}(x) =-x_{j}^{-2} ( \partial_{j}f ) ^{\ast}(x) $ we obtain
\begin{gather*}
\int_{\mathbb{T}^{N}} ( x_{j}\partial_{j}f(x) )
^{\ast}g(x) \mathrm{d}m =-\int_{\mathbb{T}^{N}}x_{j}
^{-1}\big({-}x_{j}^{2}\big) \partial_{j}f^{\ast}(x) g (x) \mathrm{d}m
 =\int_{\mathbb{T}^{N}}f^{\ast}(x) x_{j}\partial_{j}g (
x ) \mathrm{d}m.
\end{gather*}
That is, $x_{j}\partial_{j}$ is self-adjoint with respect to $\langle
f,g\rangle :=\int_{\mathbb{T}^{N}}f^{\ast}(x) g(
x) \mathrm{d}m$. The result extends to $C^{(1) }\big(
\mathbb{T}^{N};V_{\tau}\big) $.

Specialize to a closed $\mathcal{S}_{N}$-invariant subset $E\subset
\mathbb{T}_{\rm reg}^{N}$ which is the closure of its interior (for example
$E_{\varepsilon}=\big\{ x\in\mathbb{T}^{N}\colon \min_{i<j} \vert x_{i}
-x_{j}\vert \geq\varepsilon\big\} $ for $\varepsilon>0$), and let
$f,g\in C^{(1) }\big( \mathbb{T}^{N};V_{\tau}\big) $ have
supports contained in $E$, that is, $f(x) =0=g(x)
$ for $x\notin E$. For f\/ixed $f$, $g$, $i$ let
\begin{gather*}
I_{n} =\int_{\mathbb{T}^{N}}\big\{ ( x_{i}\mathcal{D}_{i}f (
x ) ) ^{\ast}K_{n}(x) g(x) -f (x ) ^{\ast}K_{n}(x) x_{i}\mathcal{D}_{i}g(x)\big\} \mathrm{d}m(x) \\
\hphantom{I_{n} }{}
 =\int_{\mathbb{T}^{N}}\big\{( x_{i}\partial_{i}f(x)
) ^{\ast}K_{n}(x) g(x) -f(x)^{\ast}K_{n}(x) x_{i}\mathcal{\partial}_{i}g(x)\big\} \\
\hphantom{I_{n}=}{}
 +\kappa\int_{\mathbb{T}^{N}}\sum_{j\neq i}\left( \tau(i,j)
x_{i}\frac{f(x) -f(x(i,j)) } {x_{i}-x_{j}}\right) ^{\ast}K_{n}(x) g(x)\mathrm{d}m(x) \\
\hphantom{I_{n}=}{}
 -\kappa\int_{\mathbb{T}^{N}}f(x) ^{\ast}K_{n}(
x) \sum_{j\neq i}\left( \tau(i,j) x_{i}\frac{g(x) -g(x(i,j)) }{x_{i}-x_{j}}\right)\mathrm{d}m(x) .
\end{gather*}
By using $\big( \frac{x_{i}}{x_{i}-x_{j}}) ^{\ast}=-\frac{x_{j}}{x_{i}-x_{j}}$, $\tau(i,j) ^{\ast}=\tau(i,j) $,
and rearranging the sums we obtain
\begin{gather*}
I_{n} =\int_{\mathbb{T}^{N}}\big\{( x_{i}\partial_{i}f(
x)) ^{\ast}K_{n}(x) g(x) -f(x) ^{\ast}K_{n}(x) x_{i}\mathcal{\partial}_{i}g(x) \big\} \mathrm{d}m(x) \\
\hphantom{I_{n} =}{}
 -\kappa\int_{\mathbb{T}^{N}}f(x) ^{\ast}\sum_{j\neq
i}\left\{ \frac{x_{j}}{x_{i}-x_{j}}\tau(i,j) K_{n} (
x ) +K_{n}(x) \tau(i,j) \frac{x_{i}}
{x_{i}-x_{j}}\right\} g(x) \mathrm{d}m(x) \\
\hphantom{I_{n} =}{}
 +\kappa\sum_{j\neq i}\int_{\mathbb{T}^{N}}\big\{
x_{j}f(x(i,j)) ^{\ast}\tau(i,j)
K_{n}(x) g(x)
+x_{i}f(x) ^{\ast}K_{n}(x) \tau(i,j)
g(x(i,j))
\big\} \frac{\mathrm{d}m(x) }{x_{i}-x_{j}}.
\end{gather*}
Each integral in the third line is f\/inite because $f$ and $g$ vanish on a
neighborhood of $\bigcup\limits_{1\leq i<j\leq N} \{ x\colon$ $x_{i}=x_{j} \} $. The
terms inside $ \{ \cdot \} $ are invariant under the change of
variable $x\mapsto x(i,j) $, because $\tau(i,j)
K_{n}(x(i,j)) =K_{n}(x) \tau(i,j) $, but the denominator $( x_{i}-x_{j}) $ changes
sign; thus the integrand is odd under~$(i,j) $ and the integral vanishes.

Because terms like $\tau(i,j) \frac{x_{i}}{x_{i}-x_{j}}g(x) $ are in $C^{(1) }\big( \mathbb{T}^{N};V_{\tau
}\big) $ (assumption on the support of~$g$) we can take the limit as
$n\rightarrow\infty$ in the second line. By the adjoint property of
$x_{i}\partial_{i}$ we f\/ind
\begin{gather*}
\int_{\mathbb{T}^{N}} ( x_{i}\partial_{i}f(x) )
^{\ast}K_{n}(x) g(x) \mathrm{d}m(x)
=\int_{\mathbb{T}^{N}}f(x) ^{\ast}x_{i}\mathcal{\partial}%
_{i} \{ K_{n}(x) g(x) \}\mathrm{d}m(x) \\
\qquad{}
=\int_{\mathbb{T}^{N}}f(x) ^{\ast}\big\{ K_{n} (
x ) x_{i}\mathcal{\partial}_{i}g(x) + ( x_{i}\partial_{i}K_{n}(x) ) g(x) \big\}
\mathrm{d}m(x) .
\end{gather*}
Thus the fact that $\lim\limits_{n\rightarrow\infty}I_{n}=0$ implies (recall the
matrix-valued integral notation~(\ref{intmat}))
\begin{gather}
 \lim_{n\rightarrow\infty}\int_{\mathbb{T}^{N}}f(x) ^{\ast
} ( x_{i}\partial_{i}K_{n}(x) ) g(x)
\mathrm{d}m(x)\nonumber\\
 \qquad{} =\kappa\int_{\mathbb{T}^{N}}f(x) ^{\ast}\sum_{j\neq i}%
\frac{1}{x_{i}-x_{j}}\big\{ x_{j}\tau(i,j) \mathrm{d}\mu(x) +\mathrm{d}\mu(x) \tau(i,j)
x_{i}\big\} g(x) .\label{bigdiff}
\end{gather}
This statement is valid for all $f,g\in C^{(1) }\big(
\mathbb{T}^{N};V_{\tau}\big) $ that vanish on a neighborhood of
$\bigcup\limits_{1\leq i<j\leq N} \{ x\colon$ $x_{i}=x_{j} \} $. The f\/irst
line of the equation can be written as
\begin{gather*}
 \lim_{n\rightarrow\infty}\int_{\mathbb{T}^{N}}\big\{ (
x_{i}\partial_{i}f(x) ) ^{\ast}K_{n}(x)
g(x) -f(x) ^{\ast}K_{n}(x)
x_{i}\mathcal{\partial}_{i}g(x) \big\} \mathrm{d}m (
x ) \\
\qquad{} =\int_{\mathbb{T}^{N}}\big\{ ( x_{i}\partial_{i}f(x)
) ^{\ast}\mathrm{d}\mu(x) g(x) -f (
x ) ^{\ast}\mathrm{d}\mu(x) x_{i}\mathcal{\partial}%
_{i}g(x) \big\},
\end{gather*}
which expresses the distributional derivative of~$\mathrm{d}\mu$. Thus the
distribution-sense dif\/ferential system is satisf\/ied by~$\mathrm{d}\mu$ on
closed subsets of $\mathbb{T}_{\rm reg}^{N}$.

In the scalar case ($\tau=(N) $) the orthogonality weight is
known to be (due to~\cite{Beerends/Opdam1993} for the symmetric Jack polynomials)
\begin{gather*}
K(x) =\prod\limits_{1\leq i<j\leq N}\big\{ ( x_{i}
-x_{j} ) \big( x_{i}^{-1}-x_{j}^{-1}\big) \big\} ^{\kappa}
\end{gather*}
and the dif\/ferential equation system reduces to (note $\tau(w)=1$)
\begin{gather*}
x_{i}\partial_{i}K(x) =\kappa K(x) \sum_{j\neq
i}\frac{x_{i}+x_{j}}{x_{i}-x_{j}}, \qquad 1\leq i\leq N.
\end{gather*}

The dif\/ferential system~(\ref{dsystK}) could be the subject for an article all
by itself, but we can sketch a result about the absolutely continuous part of
$\mu$, namely that $\mathrm{d}\mu(x) =L(x) ^{\ast
}BL(x) \mathrm{d}m(x)$, $x\in\mathbb{T}_{\rm reg}^{N}$
where $B$ is a locally constant positive matrix and~$L(x) $ is a~fundamental solution of
\begin{gather}
\partial_{i}L(x) =\kappa L(x) \left\{\sum_{j\neq i}\frac{1}{x_{i}-x_{j}}\tau(i,j) -\frac{\gamma
}{x_{i}}I\right\} , \qquad 1\leq i\leq N,\label{Lsys}\\
\gamma :=\frac{1}{2N}\sum_{i=1}^{\ell(\tau) }\tau_{i} (\tau_{i}-2i+1 ) =\frac{1}{N}\sum_{j=1}^{N}c ( j,T_{0} ).\nonumber
\end{gather}
The ef\/fect of the term $\frac{\gamma}{x_{i}}I$ is to make $L(x)
$ homogeneous of degree zero, that is, $\sum\limits_{i=1}^{N}x_{i}\partial
_{i}L(x) =0$, because
\begin{gather*}
\sum\limits_{1\leq i<j\leq N}\tau(i,j) =\left\{ \sum\limits_{j=1}^{N}c( j,T_{0}) \right\}I
\end{gather*} (the sum of the contents in the diagram of $\tau$).

The dif\/ferential system is Frobenius integrable; this means that in the system
$\partial_{i}L(x) =\kappa L(x) M_{i}(x)$, $1\leq i\leq N$, the two formal dif\/ferentiations
\begin{gather*}
\partial_{j}\partial_{i}L(x) =\kappa^{2}L(x) M_{j}(x) M_{i}(x) +\kappa L(x)\partial_{j}M_{i}(x) ,\\
\partial_{i}\partial_{j}L(x) =\kappa^{2}L(x)M_{i}(x) M_{j}(x) +\kappa L(x)\partial_{i}M_{j}(x) ,
\end{gather*}
are equal to each other for all $i$, $j$ (see \cite{Dunkl1993}). The system is
analytic and thus any local solution can be continued analytically to any
point in $\mathbb{C}_{\rm reg}^{N}:= ( \mathbb{C}\backslash \{
0 \} ) ^{N}\backslash\bigcup\limits_{1\leq i<j\leq N} \{
x\colon x_{i}=x_{j} \} $. When restricted to $\mathbb{T}_{\rm reg}^{N}$ there are
solutions def\/ined on each connected component. This is possible because
$L(x) $ is constant on $ \{ ux\colon \vert u \vert
=1 \} $ for f\/ixed $x\in\mathbb{T}_{\rm reg}^{N}$ and each component is
homotopic to a circle. Denote the component containing all the points
$ \big\{ \big( e^{\mathrm{i}\theta_{1}},\ldots,e^{\mathrm{i}\theta_{N}
}\big) \colon -\pi<\theta_{1}<\theta_{2}<\cdots<\theta_{N}<\pi\big\} $ by
$\mathcal{C}_{0}$. Since $x$ and $ux$ are in the same component for
$ \vert u \vert =1$ we see that $x_{0}\in\mathcal{C}_{0}$ (recall
$x_{0}=\big( 1,\omega,\ldots,\omega^{N-1}\big) $, $\omega=e^{2\pi
\mathrm{i}/N})$. The components are $ \{ \mathcal{C}_{0}w\colon w\in
\mathcal{S}_{N},\, w(1) =1 \} $ (corresponding to the
$( N-1) !$ circular permutations of $( 1,2,\ldots,N)
$). Now f\/ix the unique solution $L(x) $ such that $L(
x_{0}w) =I$ for each $w$ with $w(1) =1$. By
dif\/ferentiating $L(x) ^{-1}L(x) =I$ obtain the
system satisf\/ied by $L^{-1}$:
\begin{gather*}
\partial_{i}L(x) ^{-1}=-\kappa\bigg\{ \sum_{j\neq i}\frac
{1}{x_{i}-x_{j}}\tau(i,j) -\frac{\gamma}{x_{i}}I\bigg\}
L(x) ^{-1},\qquad 1\leq i\leq N.
\end{gather*}
The goal here is to replace $f$, $g$ in formula~(\ref{bigdiff}) by $L^{-1}f$, $L^{-1}g$ and deduce the desired result. By use of $x_{i}\partial_{i}(
L^{-1\ast}) =-\big( x_{i}\partial_{i}L^{-1}\big) ^{\ast}$ we obtain
\begin{gather*}
 \big( x_{i}\partial_{i}\big( L(x) ^{-1\ast}\big)
\big) \mathrm{d}\mu(x) L(x) ^{-1}+L (
x ) ^{-1\ast}\mathrm{d}\mu(x) \big( x_{i}\partial
_{i}L(x) ^{-1}\big) \\
 \qquad{} =\kappa L^{-1\ast}\bigg\{ \sum_{j\neq i}\frac{-x_{j}}{x_{i}-x_{j}}
\tau(i,j) -\gamma I\bigg\} \mathrm{d}\mu L^{-1}
 -\kappa L^{-1\ast}\mathrm{d}\mu\bigg\{ \sum_{j\neq i}\frac{x_{i}}
{x_{i}-x_{j}}\tau(i,j) -\gamma I\bigg\} L^{-1}\\
\qquad{} =-\kappa L^{-1\ast}\sum_{j\neq i}\left\{ \frac{x_{j}}{x_{i}-x_{j}}
\tau(i,j) \mathrm{d}\mu+\mathrm{d}\mu\tau(i,j)\frac{x_{i}}{x_{i}-x_{j}}\right\} L^{-1}.
\end{gather*}
Substitute this relation in formula (\ref{bigdiff})
\begin{gather*}
 \lim_{n\rightarrow\infty}\int_{\mathbb{T}^{N}}f(x) ^{\ast
}L(x) ^{-1\ast} ( x_{i}\partial_{i}K_{n}(x)
 ) L(x) ^{-1}g(x) \mathrm{d}m (x) \\
\qquad{} +\int_{\mathbb{T}^{N}}f(x) ^{\ast} ( x_{i}\partial
_{i} ) \big( L(x) ^{-1\ast}\big) \mathrm{d}\mu (x) L(x) ^{-1}g(x) \\
\qquad{} +\int_{\mathbb{T}^{N}}f(x) ^{\ast}L(x)
^{-1\ast}\mathrm{d}\mu(x) x_{i}\partial_{i}L(x)
^{-1}g(x) =0.
\end{gather*}
The formula is valid because $L^{-1}f,L^{-1}g\in C^{(1) }\big(
\mathbb{T}_{\rm reg}^{N};V_{\tau}\big) $ and vanish for $x\notin E$. The f\/irst
line of the equation is the distributional derivative of~$\mu$ so the equation is equivalent to
\begin{gather*}
\int_{\mathbb{T}^{N}}f(x) ^{\ast}\big[ x_{i}\partial
_{i}\big( L(x) ^{-1\ast}\mathrm{d}\mu(x) L(x) ^{-1}\big) \big] g(x) =0, \qquad 1\leq i\leq N.
\end{gather*}
Thus all the partial derivatives of the distribution $L^{-1\ast}\mathrm{d}\mu
L^{-1}$ vanish and $L^{-1\ast}\mathrm{d}\mu L^{-1}=B\mathrm{d}m$, where $B$ is
constant on each component of $\mathbb{T}_{\rm reg}^{N}$. We conclude
$\mathrm{d}\mu(x) =L(x) ^{\ast}BL(x)
\mathrm{d}m(x) $ for $x\in\mathbb{T}_{\rm reg}^{N}$. Part~(1) of
Theorem~\ref{propKn} implies~$B$ is a positive matrix.

We have to point out that this result provides no information about the
behavior of $\mu$ on the singular set $\bigcup\limits_{i<j} \{
x\colon x_{i}=x_{j} \} $. We conjecture that $\mu$ does not have a singular
part. This question seems a worthy topic for further investigations. Some of
the dif\/f\/iculties in this problem come from the behavior of the solutions of
system~(\ref{Lsys}) in neighborhoods of the sets $\{ x\colon x_{i}
=x_{j}\} $; there are singularities of order $ \vert x_{i}
-x_{j} \vert ^{\pm\kappa}$. Both signs appear because the eigenvalues of~$\tau(i,j) $ are~$1$ and~$-1$; a consequence of the assumption
that~$\tau$ is not one-dimensional.

\appendix
\section{Appendix}\label{appendixA}

\subsection{The matrix Bochner theorem}\label{BTheorem}

For a matrix $A\in M_{n}(\mathbb{C})$ the operator norm is%
\begin{gather*}
\left\Vert A\right\Vert :=\sup\{\vert Av \vert \colon \vert
v \vert =1 \} =\sup \{ \vert \langle u,Av \rangle \vert \colon \vert u\vert =1= \vert v \vert \} .
\end{gather*}

\begin{Proposition}
Suppose $F$ is a positive-definite $M_{n}(\mathbb{C}) $-valued
function on $\mathbb{Z}^{N}$ then
\begin{enumerate}\itemsep=0pt
\item[$(1)$] $F(\mathbf{0})$ is positive-definite,
\item[$(2)$] $F(-\alpha) =F(\alpha) ^{\ast}$ and $\Vert F(\alpha)\Vert \leq\Vert
F(\mathbf{0})\Vert $ for all $\alpha\in\mathbb{Z}^{N}$.
\end{enumerate}
\end{Proposition}

\begin{proof}
Part (1) follows immediately from taking $f(\mathbf{0}) =u$,
$f(\alpha) =0$ for $\alpha\neq0$. For part (2) f\/ix $\alpha\neq0$
and let $f(\mathbf{0}) =u$, $f(\alpha) =v$ and
$f(\beta) =0$ otherwise. By def\/inition
\begin{gather*}
u^{\ast}F(\mathbf{0}) u+v^{\ast}F(\mathbf{0})
v+v^{\ast}F(\alpha) u+u^{\ast}F(-\alpha) v\geq0.
\end{gather*}
Thus $\operatorname{Im} ( v^{\ast}F(\alpha) u+u^{\ast
}F(-\alpha) v ) =0$ for all $u$, $v$. For $1\leq j,k\leq n$
let $u=\varepsilon_{j}$, $v=c\varepsilon_{k}$, $c\in\mathbb{C}$, then
\begin{gather*}
0=\operatorname{Im}\big( \overline{c}F(\alpha) _{kj}+cF (-\alpha) _{jk}\big) .
\end{gather*}
Set $c=1$ and $c=\mathrm{i}$ to show $F(-\alpha) _{jk}=\overline{F(\alpha) _{kj}}$, that is, $F(-\alpha)
=F(\alpha) ^{\ast}$. Thus $\overline{u^{\ast}F(-\alpha) v}=v^{\ast}F(\alpha) u$ and
\begin{gather*}
-2\operatorname{Re} ( v^{\ast}F(\alpha) u ) \leq u^{\ast}F(\mathbf{0}) u+v^{\ast}F(\mathbf{0}) v.
\end{gather*}
Specialize to vectors $u$, $v$ such that $\vert u\vert =1= \vert
v \vert $ and $v^{\ast}F(\alpha) u= \langle v,F (
\alpha ) u \rangle =- \Vert F(\alpha) \Vert
$. Since $F(\mathbf{0}) $ is positive-def\/inite it follows that
$u^{\ast}F(\mathbf{0}) u+v^{\ast}F(\mathbf{0})
v\leq2 \Vert F(\mathbf{0}) \Vert $, and so $ \Vert
F(\alpha) \Vert \leq \Vert F(\mathbf{0})\Vert $.
\end{proof}

\begin{Proposition}
Suppose $\alpha,\beta\in\mathbb{Z}^{N}$ and $\alpha,\beta\neq\mathbf{0}$ then
\begin{gather*}
\Vert F(\alpha) -F(\beta) \Vert
^{2}\leq2\Vert F(\mathbf{0}) \Vert \{ v^{\ast
}F(\mathbf{0}) v-\operatorname{Re}( v^{\ast}F(
\alpha-\beta) v) \} ,
\end{gather*}
where $\vert v\vert =1$ and $\vert ( F(
\alpha) -F(\beta) ) ^{\ast}v\vert
=\Vert F(\alpha) -F(\beta) \Vert $.
\end{Proposition}

\begin{proof}
Assume $\alpha\neq\beta$ and let $f(\mathbf{0}) =u$, $f(
\alpha) =w$, $f(\beta) =-w$, and $f(\gamma)=0$ otherwise. By def\/inition
\begin{gather*}
u^{\ast}F(\mathbf{0}) u+2w^{\ast}F(\mathbf{0})
w+w^{\ast}( F(\alpha) -F(\beta) )u+u^{\ast}( F(-\alpha) -F(-\beta))w\\
\qquad{} -w^{\ast}F(\alpha-\beta) w-w^{\ast}F(\beta-\alpha) w\geq0.
\end{gather*}
Let $u,v\in\mathbb{C}^{n}$ satisfy $ \vert u \vert =1= \vert
v \vert $ and $ \vert ( F(\alpha) -F (\beta ) ) ^{\ast}v \vert =v^{\ast} ( F (\alpha ) -F(\beta) ) u = \Vert
F(\alpha) -F(\beta) \Vert $. Set $w=tv$ with $t\in\mathbb{R}$. The inequality becomes
\begin{gather*}
u^{\ast}F( \mathbf{0}) u+2t^{2}v^{\ast}F( \mathbf{0}t) v+2t\Vert F(\alpha) -F(\beta)
\Vert -2t^{2}\operatorname{Re}\big( v^{\ast}F ( \alpha
-\beta ) v\big) \geq0.
\end{gather*}
The discriminant of the quadratic polynomial in $t$ must be nonpositive and
this implies the stated inequality, since $u^{\ast}F( \mathbf{0})
u\leq\Vert F( \mathbf{0})\Vert $.
\end{proof}

This is a sort of uniform continuity.

For any f\/ixed $u\in\mathbb{C}^{n}$ the scalar function $g_{u}(\alpha) =u^{\ast}F(\alpha) u$ is positive-def\/inite in the
classical sense and thus by Bochner's theorem (see \cite[pp.~17--21]{Rudin1962}) there exists a unique positive Baire measure~$\mu_{u}$ on~$\mathbb{T}^{N}$ such that
\begin{gather*}
\int_{\mathbb{T}^{N}}x^{-\alpha}\mathrm{d}\mu_{u}(x) =u^{\ast}F(\alpha) u, \qquad \alpha\in\mathbb{Z}^{N}.
\end{gather*}
In particular, for $1\leq j\leq n$ let $u=\varepsilon_{j}$ then $u^{\ast
}F(\alpha) u=F(\alpha) _{jj}$ and set $\mu
_{jj}=\mu_{\varepsilon_{j}}$. For $1\leq j,k\leq n$ (with $j\neq k$) let
$u=\varepsilon_{j}+\varepsilon_{k}$, $v=\varepsilon_{j}+\mathrm{i}\varepsilon_{k}$
\begin{gather*}
\int_{\mathbb{T}^{N}}x^{-\alpha}\mathrm{d}\mu_{u}(x) =F(\alpha) _{jj}+F(\alpha) _{jk}+F\left(
\alpha\right) _{kj}+F(\alpha) _{kk},\\
\int_{\mathbb{T}^{N}}x^{-\alpha}\mathrm{d}\mu_{v}(x)=F(\alpha) _{jj}+\mathrm{i}F(\alpha)
_{jk}-\mathrm{i}F(\alpha) _{kj}+F(\alpha) _{kk}.
\end{gather*}
Thus
\begin{gather*}
\int_{\mathbb{T}^{N}}x^{-\alpha}\mathrm{d} ( \mu_{u}-\mathrm{i}\mu
_{v} ) =2F(\alpha) _{jk}+ ( 1-\mathrm{i} )
\int_{\mathbb{T}^{N}}x^{-\alpha}\mathrm{d} ( \mu_{jj}+\mu_{kk} ) ,
\end{gather*}
and def\/ine
\begin{gather*}
\mu_{jk}=\frac{1}{2} ( \mu_{u}-\mathrm{i}\mu_{v} ) -\frac{1-\mathrm{i}}{2} ( \mu_{jj}+\mu_{kk} ) ,
\end{gather*}
with the result
\begin{gather*}
\int_{\mathbb{T}^{N}}x^{-\alpha}\mathrm{d}\mu_{jk}(x) =F (\alpha) _{jk},\alpha\in\mathbb{Z}^{N}.
\end{gather*}
From the above equations we also obtain%
\begin{gather*}
\int_{\mathbb{T}^{N}}x^{-\alpha}\mathrm{d} ( \mu_{u}+\mathrm{i}\mu
_{v} ) =2F(\alpha) _{kj}+ ( 1+\mathrm{i} )
\int_{\mathbb{T}^{N}}x^{-\alpha}\mathrm{d} ( \mu_{jj}+\mu_{kk}) .
\end{gather*}
The formula
\begin{gather*}
\mu_{kj}=\frac{1}{2} ( \mu_{u}+\mathrm{i}\mu_{v} ) -\frac
{1+\mathrm{i}}{2} ( \mu_{jj}+\mu_{kk} ) ,
\end{gather*}
is consistent with the previous one and it can be directly verif\/ied that
$\widehat{\mu_{kj}}(-\alpha) =\overline{\widehat{\mu_{jk}
}(\alpha) }$ from the general equation $\widehat{\mu_{u}} (
-\alpha ) =\overline{\widehat{\mu_{u}}(\alpha) }$. This is
a restatement of $F(-\alpha) =F(\alpha) ^{\ast}$.

From $\Vert \mu_{u} \Vert =u^{\ast}F(\mathbf{0}) u$
(measure/total variation norm) we obtain $\Vert \mu_{jj}\Vert =F(\mathbf{0}) _{jj}$ and
\begin{gather*}
\Vert \mu_{jk}\Vert \leq\frac{1}{\sqrt{2}}( F(\mathbf{0}) _{jj}+F(\mathbf{0}) _{kk}) +F(
\mathbf{0}) _{jj}+F(\mathbf{0}) _{kk}+\vert\operatorname{Re}F(\mathbf{0}) _{jk}\vert +\vert
\operatorname{Im}F(\mathbf{0}) _{jk}\vert .
\end{gather*}
In particular, if $F(\mathbf{0}) =I$ then $\Vert \mu
_{jk}\Vert \leq2+\sqrt{2}$. For the space $C\big( \mathbb{T}^{N};\mathbb{C}^{n}\big) $ (continuous functions on $\mathbb{T}^{N}$,
values in $\mathbb{C}^{n}$) def\/ine an inner product
\begin{gather*}
\langle f,g\rangle _{F}:=\sum_{i,j=1}^{n}\int_{\mathbb{T}^{N}}\overline{f(x) _{i}}g(x) _{j}\mathrm{d}\mu_{ij}(x) ,
\end{gather*}
then for $\alpha,\beta\in\mathbb{Z}^{N}$, and $1\leq i,j\leq n$ (with the
standard unit basis vectors $\varepsilon_{i}$ of $\mathbb{C}^{n}$)
\begin{gather*}
\big\langle x^{\alpha}\varepsilon_{i},x^{\beta}\varepsilon_{j}\big\rangle
_{F}=\int_{\mathbb{T}^{N}}x^{-\alpha}x^{\beta}\mathrm{d}\mu_{ij}(x) =F(\alpha-\beta) _{ij}.
\end{gather*}
The following summarizes the above results.

\begin{Theorem}\label{Bformthm}The Hermitian form $\langle \cdot,\cdot\rangle
_{F}$ is bounded on $C\big( \mathbb{T}^{N};\mathbb{C}^{n}\big) $, that is
$\vert \langle f,g\rangle _{F}\vert \leq B\Vert
f\Vert _{\infty}\Vert g\Vert _{\infty}$ for some $B<\infty$,
is positive-semidefinite, and if $f$, $g$ are finitely supported functions on
$\mathbb{Z}^{N}$ with values in $\mathbb{C}^{n}$ then for $\widehat{f}(x) :=\sum_{\alpha}f(\alpha) x^{\alpha}$ and $\widehat
{g}(x) :=\sum_{\alpha}g(\alpha) x^{\alpha}$
\begin{gather*}
\big\langle \widehat{f},\widehat{g}\big\rangle _{F}=\sum_{\alpha,\beta
}f(\alpha) ^{\ast}F(\alpha-\beta) g(\beta) .
\end{gather*}
\end{Theorem}

\begin{proof}
The bound follows from the uniform bound on $\Vert \mu_{jk}\Vert $
for all $j$, $k$, depending only on~$F(\mathbf{0}) $. Suppose
$\widehat{f}$, $\widehat{g}$ are trigonometric (Laurent) polynomials (equivalent
to f\/inite support on~$\mathbb{Z}^{N}$), then
\begin{gather*}
\big\langle \widehat{f},\widehat{g}\big\rangle _{F} =\sum_{\alpha
,\beta}\sum_{i,j=1}^{n}\int_{\mathbb{T}^{N}}\overline{f(\alpha)
_{i}}x^{-\alpha}g(\beta) _{j}x^{\beta}\mathrm{d}\mu_{ij} (x) \\
\hphantom{\big\langle \widehat{f},\widehat{g}\big\rangle _{F}}{}
 =\sum_{\alpha,\beta}\sum_{i,j=1}^{n}\overline{f(\alpha) _{i}}F(\alpha-\beta) _{ij}g(\beta) _{j}=\sum
_{\alpha,\beta}f(\alpha) ^{\ast}F(\alpha-\beta)g(\beta) .
\end{gather*}
By def\/inition $\big\langle \widehat{f},\widehat{f}\big\rangle _{F}\geq0$
and by the density of the trigonometric polynomials in $C\big(
\mathbb{T}^{N};\mathbb{C}^{n}\big) $ it follows that $\langle
\cdot,\cdot\rangle _{F}$ is a positive semidef\/inite (it is possible that
$\big\langle \widehat{f},\widehat{f}\big\rangle _{F}=0$ for some $f\neq0$)
Hermitian form.
\end{proof}

The next result is used for the approximate identity arguments.

\begin{Theorem}\label{posdefpt} Suppose $\sigma\in C\big( \mathbb{T}^{N}\big) $ and
$\sigma(x) \geq0$ for all $x\in\mathbb{T}^{N}$ then each
$\sigma\ast\mu_{ij}\in C\big( \mathbb{T}^{N}\big) $ and $[\sigma\ast\mu_{ij}(x)] _{i,j=1}^{n}\in M_{n}(
\mathbb{C}) $ is positive semidefinite for each $x\in\mathbb{T}^{N}$.
\end{Theorem}

\begin{proof}
Suppose $( f_{i}) _{i=1}^{n}\in C\big( \mathbb{T}^{N};\mathbb{C}^{n}\big)$. A convolution formula (similar to that in Lemma~\ref{convghnu}) shows that
\begin{gather*}
 \sum_{i,j=1}^{n}\int_{\mathbb{T}^{N}}\overline{f(x) _{i}}\sigma\ast\mu_{ij}(x) f(x) _{j}\mathrm{d}m(x)
 =\sum_{i,j=1}^{n}\int_{\mathbb{T}^{N}}\int_{\mathbb{T}^{N}}\sigma\big(
xy^{-1}\big) \overline{f(x) _{i}}f(x) _{j}\mathrm{d}\mu_{ij}(y) \mathrm{d}m(x) \\
\qquad{} =\int_{\mathbb{T}^{N}}\sigma(u) \int_{\mathbb{T}^{N}}%
\sum_{i,j=1}^{n}\overline{f(yu) _{i}}f(yu)
_{j}\mathrm{d}\mu_{ij}(y) \mathrm{d}m(u) \geq0,
\end{gather*}
(change-of-variable $x=yu$) because $y\mapsto f(yu) \in C\big(
\mathbb{T}^{N};\mathbb{C}^{n}\big) $ and the double sum is continuous in
$u$ and nonnegative by the above theorem. Now let $f(x)
=g(x) v$ where $v\in\mathbb{C}^{n}$ and $g\in C\big(
\mathbb{T}^{N}\big) $, $g\geq0$, $\int_{\mathbb{T}^{N}}g^{2}\mathrm{d}m=1$
and $g=0$ of\/f an $\varepsilon$-neighborhood of a f\/ixed $z\in\mathbb{T}^{N}$
(i.e., $\vert x-z\vert \geq\varepsilon$ implies $g(x)=0$) then
\begin{gather*}
 \sum_{i,j=1}^{n}\int_{\mathbb{T}^{N}}\overline{f(x) _{i}}\sigma\ast\mu_{ij}(x) f(x) _{j}\mathrm{d}m(x)
 =\int_{\mathbb{T}^{N}}g(x) ^{2}\sum_{i,j=1}^{n}\overline{v_{i}}( \sigma\ast\mu_{ij}(x)) v_{j} \mathrm{d}m(x) \geq0.
\end{gather*}
Let $\varepsilon\rightarrow0$ then the integral tends to $\sum\limits_{i,j=1}^{n}\overline{v_{i}}
( \sigma\ast\mu_{ij} ( z ) )v_{j}$, and this completes the proof.
\end{proof}

\subsection{Some results for the index set}

For $n=1,2,3,\ldots$
\begin{gather}
\#\boldsymbol{Z}_{N,n}=\sum_{j=1}^{N-1}\binom{N}{j}\binom{n-1}{j-1}%
\binom{N-j+n-1}{n}; \label{ZNnct}
\end{gather}
in each $j$-subset (there are $\binom{N}{j}$) of $[ 1,2,\ldots,N]
$ take $j$-tuples $\alpha_{i_{1}},\ldots,\alpha_{i_{j}}$ with $\sum\limits_{\ell
=1}^{j}\alpha_{i_{\ell}}=n$ and each $\alpha_{i_{\ell}}\geq1$ ($\binom
{n-1}{j-1}$ possibilities) and in the complement take $(N-j)
$-tuples with $\sum\limits_{\ell=1}^{N-j}\alpha_{i_{\ell}}=-n$ and each
$\alpha_{i_{\ell}}\leq0$ ($\binom{N-j+n-1}{n}$ possibilities). For example
when $n\geq1$
\begin{gather*}
\#\boldsymbol{Z}_{2,n}=2, \quad \#\boldsymbol{Z}_{3,n}=6n, \quad \#\boldsymbol{Z}
_{4,n}=10n^{2}+2, \quad \#\boldsymbol{Z}_{5,n}=\frac{5n}{3}\big( 7n^{2}+5\big), \quad \ldots.
\end{gather*}

\begin{Proposition}\label{mingam}Among $\gamma\in\boldsymbol{Z}_{N,n}$ such that $\gamma^{\nu
}=\beta$ for some fixed $\beta$ with $ \vert \beta \vert =n$ and such
that $\beta_{j}>0$ exactly when $j>k$ the minimal multi-index for the order
$\gamma^{(1) \pi}\vartriangleright\gamma^{(2) \pi
}$ is $\gamma^{(0)}= \big( p+1,\ldots,\overset{(m) }{p+1},p,\ldots,\overset{(k) }{p},-\beta_{k+1},\ldots,-\beta_{N}\big) $ where $p=\big\lfloor \frac{n}{k}\big\rfloor $ and $m=n-kp$ $($so $0\leq m<k)$.
\end{Proposition}

\begin{proof}
The claim is that $\big( \gamma_{i}^{(0) }\big) _{i=1}^{k}$
is $\prec$-minimal among partitions~$\alpha$ of length $\leq k$ and
$\vert\alpha\vert =n$. Argue by induction on the length. The
statement is obviously true when $k=1$. Suppose it is true for~$k$ and let
$\alpha$ be a partition of length $\leq k+1$. Let
\begin{gather*}
n_{1} =\sum_{i=1}^{k}\alpha_{i}, \qquad n_{2}=n_{1}+\alpha_{k+1}, \qquad p_{1}=\left\lfloor \frac{n_{1}}{k}\right\rfloor ,
\qquad p_{2}=\left\lfloor \frac{n_{2}}{k+1}\right\rfloor ,\\
m_{1} =n_{1}-kp_{1}, \qquad m_{2}=n_{2}-(k+1) p_{2}.
\end{gather*}
Def\/ine $\gamma^{(1) }$ and $\gamma^{(2) }$
analogously to the above ($\gamma_{i}^{(s) }=p_{s}+1$ for $1\leq
i\leq m_{s}$ and $\gamma_{i}^{(1) }=p_{1}$ for $m_{1}<i\leq k$,
$\gamma^{(2) }=p_{2}$ for $m_{2}<i\leq k+1$). By the inductive
hypothesis $\sum\limits_{j=1}^{i}\alpha_{j}\geq\sum\limits_{j=1}^{i}\gamma_{j}^{(1)}$. This implies $\alpha_{k+1}\leq\alpha_{k}\leq\gamma_{k}^{(1) }$. Thus $(k+1) p_{2}\leq n_{1}+\alpha_{k+1}\leq m_{1}+kp_{1}+\alpha_{k+1}\leq m_{1}+(k+1) p_{1}$ and $p_{2}
\leq\frac{m_{1}}{k+1}+p_{1}$. Since $p_{1}$, $p_{2}$ are integers this implies
$p_{2}\leq p_{1}$. If $p_{1}=p_{2}$ then $m_{2}=m_{1}-( p_{1}-\alpha_{k+1}) $, and clearly $\sum\limits_{j=1}^{i}\gamma_{j}^{(1) }\geq\sum\limits_{j=1}^{i}\gamma_{j}^{(2) }$ for $1\leq i\leq k$. If $p_{2}<p_{1}$ then $\gamma_{j}^{(2) }\leq p_{2}+1\leq
p_{1}\leq\gamma_{j}^{(1) }$ for $1\leq j\leq k$. Thus $\alpha\succeq\gamma^{(2) }$.
\end{proof}

\pdfbookmark[1]{References}{ref}
\LastPageEnding

\end{document}